\documentclass[11pt,a4paper,english,reqno]{amsart}
\usepackage{amsmath,amssymb,amsfonts,epsfig,mathrsfs}
\usepackage[T1]{fontenc}

\usepackage{color}
\usepackage{array}
\usepackage{amsthm}
\usepackage{amstext}
\usepackage{graphicx}
\usepackage{setspace}
\usepackage[margin=2.5cm]{geometry}
\usepackage{bbm}
\usepackage{color}
\usepackage{enumitem}
\allowdisplaybreaks[4]

\usepackage{amscd,psfrag}
\usepackage{yhmath}
\usepackage[mathscr]{eucal}
\usepackage{slashed}
\usepackage{verbatim}
\makeatletter
\pdfpageheight\paperheight
\pdfpagewidth\paperwidth

\usepackage{textcomp}

\usepackage{epstopdf}
\usepackage{chngcntr}
\usepackage{mathrsfs}

\usepackage{comment}
\usepackage{indentfirst}	

\usepackage[normalem]{ulem}
\theoremstyle{plain}

\newtheorem{definition}{Definition}[section]
\newtheorem{theorem}[definition]{Theorem}
\newtheorem*{theorem*}{Theorem}

\newtheorem{remark}[definition]{Remark}

\newtheorem*{remark*}{Remark}
\newtheorem*{sideremark*}{Side Remark}\newtheorem*{mt*}{Main Theorem}

\newtheorem*{claim*}{Claim}
\newtheorem*{q*}{Question}
\newtheorem{lemma}[definition]{Lemma}
\newtheorem{corollary}[definition]{Corollary}
\newtheorem*{corollary*}{Corollary}
\newtheorem*{proposition*}{Proposition}

\newtheorem{proposition}[definition]{Proposition}

\newcommand{\R}{\mathbb{R}}

\newcommand{\C}{\mathbb{C}}
\newcommand{\na}{\nabla}

\newcommand{\dd}{{\rm d}}
\newcommand{\p}{\partial}
\newcommand{\e}{\epsilon}

\newcommand{\G}{\Gamma}

\newcommand{\B}{\mathbf{B}}

\newcommand{\bb}{{\mathfrak{b}}}
\newcommand{\qq}{{\mathfrak{q}}}

\newcommand{\mm}{{\mathfrak{m}}}
\newcommand{\nn}{{\mathfrak{n}}}
\newcommand{\proj}{\mathbf{Proj}}
\newcommand{\sph}{{\mathbf{S}^{d-1}}}

\def\Xint#1{\mathchoice
{\XXint\displaystyle\textstyle{#1}}%
{\XXint\textstyle\scriptstyle{#1}}%
{\XXint\scriptstyle\scriptscriptstyle{#1}}%
{\XXint\scriptscriptstyle\scriptscriptstyle{#1}}%
\!\int}
\def\XXint#1#2#3{{\setbox0=\hbox{$#1{#2#3}{\int}$ }
\vcenter{\hbox{$#2#3$ }}\kern-.6\wd0}}

\def\dashint{\Xint-}

\newcommand{\maino}{\mathring{J}_{\frac{d}{2}-1}^n}
\newcommand{\maine}{\mathring{J}_{\frac{d}{2}}^n}

\newcommand{\remo}{\mathscr{R}_{\frac{d}{2}-1}^n}
\newcommand{\reme}{\mathscr{R}_{\frac{d}{2}}^n}

\title{Expected signature of stopped Brownian motion on $d$-dimensional $C^{2, \alpha}$-domains has finite radius of convergence everywhere: $2\leq d \leq 8$}

\address{Siran Li: School of Mathematical Sciences, Shanghai Jiao Tong University, No.~6 Science Buildings,
800 Dongchuan Road, Minhang District, Shanghai, China (200240)}

\address{Siran Li: Key Laboratory of Scientific and Engineering Computing (Ministry of Education), Shanghai Jiao Tong University, No.~6 Science Buildings,
800 Dongchuan Road, Minhang District, Shanghai, China (200240)}

\author{Siran Li}
\address{Siran Li: New York University -- Shanghai, Office 1146, 1555 Century Avenue, Pudong District, Shanghai, China (200122)}

\address{Siran Li: NYU-ECNU Institute of Mathematical Sciences, Room 340, Geography Building, 3663 North Zhongshan Road, Shanghai, China (200062)}

\email{\texttt{siran.li@sjtu.edu.cn}}

\author{Hao Ni}\address{Hao Ni: Department of Mathematics University College, Gower Street, WC1E 6BT, London, UK}
\email{\texttt{h.ni@ucl.ac.uk}}



\date{\today}

\begin{document}

\begin{abstract}
A fundamental question in rough path theory is whether the expected signature of a geometric rough path completely determines the law of signature. One sufficient condition is that the expected signature has infinite radius of convergence, which is satisfied by various stochastic processes on a fixed time interval, including the Brownian motion. In contrast, for the Brownian motion stopped upon the first exit time from a bounded domain $\Omega$, it is only known that the radius of convergence for the expected signature on sufficiently regular $\Omega$ is strictly positive everywhere, and that  the radius of convergence is finite at some point when $\Omega$ is the $2$-dimensional unit disc  (\cite{boedihardjo2021expected}).

In this paper, we prove that on any bounded $C^{2,\alpha}$-domain $\Omega \subset \R^d$ with $2\leq d \leq 8$, the expected signature of the stopped Brownian motion has finite radius of convergence everywhere. 
A key ingredient of our proof is the introduction of a ``domain-averaging hyperbolic development'' (see Definition~\ref{def: domain avg}), which allows us to symmetrize the PDE system for the hyperbolic development of expected signature by averaging over rotated domains. 
\end{abstract}

\maketitle

\section{Introduction}

The theory of rough paths is a generalisation of the classical control theory. It sets forth a framework that makes sense of solutions to differential equations driven by irregular signals; for example, the Brownian motion.

A fundamental concept in rough path theory is the \emph{signature} of a path. 
Playing a similar role to that of the moment generating function of a random variable, the \emph{expected signature} is of both theoretical and practical significance. For example, computation for the expected signature of the Brownian motion on $[0, 1]$ leads to the notion of \emph{cubature on Wiener space}, which is a high order numerical method for high dimensional SDEs (stochastic differential equations) and semi-elliptic PDEs (partial differential equations); see \cite{lyons2004cubature}. Recently, in \cite{chevyrev2018signature} a metric for laws of stochastic processes has been proposed based on the normalised expected signature. It can be used for nonparametric two-sample hypothesis tests for laws of stochastic processes, and potentially has much wider applications in machine learning.

One central question concerning the expected signature, known as the \emph{moment problem}, asks if the expected signature uniquely determines the law of random signature. A sufficient condition for the affirmative answer has recently been identified for the moment problem by Chevyrev--Lyons (\cite{cl}, Proposition~6.1): the radius of convergence of expected signature is infinite. 

The infiniteness of radius of convergence for the expected signature has been verified for several popular stochastic processes on a fixed time horizon, \emph{e.g.}, fractional Brownian motions with Hurst parameter in $\left]\frac{1}{4}, 1\right]$ (see \cite{cl, passeggeri2016some}). It is nonetheless considerably challenging to check whether this is satisfied by processes \emph{up to a random time}. Even for the simplest case of the Brownian motion up to the first exit time from a bounded domain $\Omega \subset \R^d$ (in brief, ``\emph{stopped Brownian motion}'' in the sequel), it has remained open ever since it was proposed in \cite{lyons2015expected}. 

As a first step towards the aforementioned problem, it was established in \cite{lyons2015expected} that, under suitable regularity assumptions for the domain (see Proposition~\ref{prop: schauder} and Remark~\ref{rem: elliptic}), the expected signature of the stopped Brownian motion has a geometric upper bound for its decay rate. This is achieved by deriving a system of nested PDEs  (\emph{i.e.}, partial differential equations graded by a parameter $n 
\in \mathbb{N}$) satisfied by the expected signature and applying the standard boundary regularity theory for elliptic PDEs. Despite its insufficiency to resolve the moment problem, the geometric upper bound warrants the positivity of the radius of convergence.

Recently, Boedihardjo--Diehl--Mezzarobba--Ni \cite{boedihardjo2021expected} showed that the  stopped Brownian motion on the $2$-dimensional unit disc has finite radius of convergence. The proof in \cite{boedihardjo2021expected} relies crucially  on the technique of \emph{hyperbolic development}, which plays an essential role in  Hambly--Lyons' proof  of the unique determination of paths of bounded variation by the signature; see  \cite{hl}. Indeed, by taking a morphism (which appears in \cite{hl}) from the tensor algebra space to  $\mathfrak{gl}(3;\R)$  and exploiting the partial symmetries of the resulting matrix-valued PDE systems,  a nested system of three ODEs has been obtained in \cite{boedihardjo2021expected}. These ODEs can be solved explicitly, from which one deduces the finiteness of radius of convergence. Nonetheless,  \cite{boedihardjo2021expected} relies heavily on the rotational invariance of the domain, thus making it difficult to be extended to general domains. In addition, even for the case of the unit disc, it remains unknown if the stopped Brownian motion starting at a point apart from the centre has finite radius of convergence.


The main result of our paper gives a complete solution to the finiteness problem of the expected signature of stopped Brownian motions on bounded domains under mild regularity assumptions, for all dimensions up to $8$:
\begin{theorem}\label{thm: main}
Let $\Omega$ be bounded  $C^{2,\alpha}$-domain in $\R^d$; $2 \leq d \leq 8$. The expected signature $\Phi$ of a Brownian motion stopped upon the first exit time from $\Omega$ has finite radius of convergence at every point in $\Omega$. 
\end{theorem}

The strategy of our proof is outlined as follows. First, by averaging the hyperbolic development of expected signature over the rotated images of the domain, we construct a quantity $\overline{\mathcal{H}_{\lambda, \e}}(z)$, the ``\emph{domain-averaging development}'', that has the following features:
\begin{itemize}
    \item 
$\overline{\mathcal{H}_{\lambda, \e}}(z)$ is rotationally  invariant (with respect the domain rotations);
\item
$\overline{\mathcal{H}_{\lambda, \e}}(z)$ satisfies the same PDE as that for $\mathcal{H}_{\lambda, \Omega}$; and
\item
$\overline{\mathcal{H}_{\lambda, \e}}(z)$ inherits the finiteness of radius of convergence from the non-averaged hyperbolic development $\mathcal{H}_{\lambda, \Omega}$ of the expected signature. Thus, the operation of domain averaging preserves the lower bound for radius of convergence of the hyperbolic development. 
\end{itemize}
Working locally near $z \in \Omega$, we arrive at the same PDE for $\overline{\mathcal{H}_{\lambda, \e}}(z)$ as for its  non-averaged analogue on a small ball around $z$. Loosely speaking, ``half'' of the boundary conditions will gone missing for the PDE for $\overline{\mathcal{H}_{\lambda, \e}}(z)$; nevertheless, geometric properties of the hyperbolic development enable us to establish uniform lower bounds for the remaining component of $\overline{\mathcal{H}_{\lambda, \e}}(z)$.


\noindent
{\bf Organisation.} The remaining parts of the paper is organised as follows. \S\ref{sec: prelim} summarises background materials on the rough path theory, expected signature, and hyperbolic development. In \S\ref{sec: lemmata} we discuss the regularity theory and  probabilistic consequences of the PDE associated to the expected signature. The crucial technique in this work (a symmetrization argument) shall be introduced in \S\ref{sec: symm}. Our main result, Theorem~\ref{thm: main}, is proved in \S\ref{sec: proof of main thm} for $d=2$ and in \S\ref{sec: high d} for $3\leq d \leq 8$. Some technical computations are given in the appendix.

We denote $E := \R^d$ throughout this paper.

\section{Preliminaries}\label{sec: prelim}

\subsection{Signature} We first introduce the tensor algebra space over $E$:
\begin{definition}
A formal $E$-tensor series is a sequence of tensors $\left( a_{n}\in
E^{\otimes n}\right) _{n\in \mathbb{N}}$ which we write as $a=\left(
a_{0},a_{1},\ldots \right) $. There are two binary operations on $E$-tensor series, addition $+$ and product $\otimes$ --- let $\mathbf{a}=(a_{0},a_{1},\ldots)$ and $\mathbf{b}%
=(b_{0},b_{1},\ldots)$ be $E$-tensor series; then $
\mathbf{a}+\mathbf{b}:=(a_{0}+b_{0},a_{1}+b_{1},\ldots)$ and $\mathbf{a}\otimes \mathbf{b}:=(c_{0},c_{1},\ldots)$, where $c_{n}:=\sum_{k=0}^{n}a_{k}\otimes b_{n-k}$ for each $n$.

In addition, we write $\mathbf{1}:=(1,0,\ldots)$, $\mathbf{0}:=(0,0,\ldots)$, and $\lambda \mathbf{a}:=(\lambda a_{0},\lambda a_{1},\ldots)$ for $\lambda \in \mathbb{R}$. Also, $T\left( \left( E\right) \right)$ denotes the vectorspace of formal $E$-tensor series.
\end{definition}

The space $T\left( \left( E\right) \right) $ equipped with $+$ and $\otimes $ is an associative unital algebra over $\mathbb{R}
$. An element  $\mathbf{a}=(a_{0},a_{1},\ldots) \in T\left( \left( E\right)
\right) $ is invertible if and only if $a_{0}\neq 0$. In this case, its inverse is 
\begin{equation*}
\mathbf{a}^{-1}=\frac{1}{a_{0}}\sum_{n\geq 0}\left(\mathbf{1}-\frac{\mathbf{a}}{%
a_{0}}\right)^{n}.
\end{equation*}%
It is well defined because, at any given degree, only finitely many terms of the summation are non-zero. In particular, the
subset $\left\{\mathbf{a}\in T\left( \left( E\right) \right) :\,a_{0}=1\right\}$ forms a group.

\begin{definition}
Let $n\geq 1$ be an integer. Let $B_{n}:=\{\mathbf{a}
=(a_{0},a_{1},\ldots):\,a_{0}=\ldots=a_{n}=0\}.$ The truncated tensor algebra $%
T^{(n)}(E)$ of order $n$ over $E$ is defined as the quotient algebra
\begin{equation*}
T^{(n)}(E)=T\left( \left( E\right) \right) /B_{n}.
\end{equation*}%
The canonical epimorphism $T\left( \left( E\right) \right) \rightarrow
T^{(n)}(E)$ is denoted by $\proj_{n}$.
\end{definition}

We collect below a few basic facts about geometric rough paths. See \cite{lyons2007differential, coutin2002stochastic, friz2006note} for details.

For $p \geq 1$ and $J \subset \R$ an open interval, one can enhance a path $X: J \rightarrow E$ of finite $1$-variation to a function defined on the simplex $\left\{(s, t):\, s \leq t, s,t \in J\right\}$ with values in $T^{(\lfloor p \rfloor)}(E)$ via its iterated integrals; \emph{i.e.}, $S^{\lfloor p \rfloor}(X) = \left(1, X^{1}_{J}, \cdots X^{\lfloor p \rfloor}_{J}\right)$ where
\begin{eqnarray*}
X^{n}_{s, t} = \underset{s \leq u_1 \leq \cdots \leq u_n \leq t }{\int \cdots \int} \,\dd X_{u_{1}} \otimes \,\dd X_{u_{2}}\otimes \cdots \otimes  \,\dd X_{u_{n}},
\end{eqnarray*}
for $n = 1, 2, \cdots, \lfloor p \rfloor$. We call the enhancement of $X$ in $T^{(\lfloor p \rfloor)}(E)$ a \emph{smooth rough path}. The signature of $X$ is the collection $S(X) = \left(1, X^{1}_{J}, \cdots, X^{n}_{J}, \cdots\right)$ of all the iterated integrals of $X$.
\begin{definition}[$p$-variation distance]
Let $p \geq 1$. For smooth rough paths $\mathbf{X}, \mathbf{Y}:J \to T^{(\lfloor p \rfloor)}(E)$, the $p$-variation metric between $\mathbf{X}$ and $\mathbf{Y}$ is given by
\begin{eqnarray*}
d(\mathbf{X}, \mathbf{Y}) = \max_{k = 1, \cdots \lfloor p \rfloor} \left(\sup_{\mathcal{D}}\sum_{i = 1}^{n-1} \left\vert \mathbf{X}_{u_i, u_{i+1}}^{k} - \mathbf{Y}_{u_{i}, u_{i+1}}^{k} \right\vert^{\frac{p}{k}}\right)^{\frac{k}{p}} + \sup_{t \in J} \left\vert X^{1}_t - Y^{1}_{t}\right\vert,
\end{eqnarray*}
with the supremum taken over all finite divisions $\mathcal{D} = (u_{1}, u_{2}, \cdots, u_{n})$ of $J$.
\end{definition}

\begin{definition}[Geometric rough path]
$\mathbf{X}: J \rightarrow T^{(\lfloor p \rfloor)}(E)$ is said to be a geometric $p$-rough path if there exists a sequence of smooth rough path $(\mathbf{X}^{n})_{n\in\mathbb{N}}$ such that $\mathbf{X}$ is the limit of $\mathbf{X}^{n}$ in the $p$-variation metric. 
The space of geometric $p$-rough paths is denoted as $G\Omega_p(J, E)$.
\end{definition}

\subsection{Stopped Brownian motion and its expected signature}\label{sec: expected signature}
Denote by $(e_{1}, \ldots, e_d)$ the canonical basis for $E=\mathbb{R}^{d}$. Let $\left( B_{t}\right) _{t\geq 0}$ be a standard Brownian motion on $E$ under a probability space $(\mathbf{\Omega}, \mathcal{F}, P^{z})$ with its canonical filtration $\mathcal{F} = \left(\mathcal{F}_{t}\right)_{t\geq 0}$, where $P^{z}(B_{0} = z) = 1$ for $z \in E$.

\begin{definition}
Let $\Omega $ be a domain (\emph{i.e.}, a connected open set) in $E$. Then $$\tau
_{\Omega }=\inf \{t\geq 0:B_{t}\in E\setminus\Omega\}$$ is the first exit time of
Brownian motion from $\Omega$.
\end{definition}

\begin{definition}
Let $J$ be a compact time interval. Let $B: J \rightarrow E$ be an  $E$-valued Brownian motion path. The signature of $B$, denoted by $S(B_{J})$, is the element $(1,B^{1},...,B^{n},...)$ of $T\left( \left( E\right)
\right) $ defined for each $n\geq 1$ as follows:%
\begin{equation*}
B^{n}_{J}=\underset{\underset{u_{1},\ldots,u_{n}\in J}{u_{1}<\ldots<u_{n}}}{\int \cdots \int }\,\dd B_{u_{1}}\otimes \cdots \otimes \,\dd B_{u_{n}},
\end{equation*}
where the integral is taken in the Stratonovich sense. The truncated signature of $B$ of order $n$ is denoted by $S^{n}(B_{J})$, \emph{i.e.}, $%
S^{n}(B_{J}) = \proj_{n}(S(B_{J}))$.
\end{definition}

One also writes
\begin{equation*}
\mathbf{B}_{J}=\sum_{n=0}^{\infty }\sum_{i_{1},\ldots i_{n}\in \{1,\ldots
d\}}\left( \underset{\underset{u_{1},\ldots,u_{n}\in J}{u_{1}<\ldots<u_{n}}}{\int
\cdots \int }\,\dd B_{u_{1}}^{(i_{1})}\otimes \ldots \otimes
\,\dd B_{u_{n}}^{(i_{n})}\right) e_{i_{1}}\otimes \ldots \otimes e_{i_{n}}
\end{equation*}
where $$ \underset{\underset{u_{1},\ldots,u_{n}\in J}{u_{1}<\ldots<u_{n}}}{\int
\cdots \int }\,\dd B_{u_{1}}^{(i_{1})}\otimes \ldots \otimes\,\dd B_{u_{n}}^{(i_{n})} =: \proj^{I}(S(B_{J}))$$ is the co-ordinate signature of Brownian motion indexed by $I:= (i_{1}, \ldots, i_{n})$. The (Stratonovich) signature is defined for \emph{a.e.} Brownian path $B$ and for all pairs of times $(s, t)$ where $s \leq t$. 

\begin{lemma}[See {\cite{friz2005approximations}}]\label{Lemma_BM_GRP}
For $T >0$, the Stratonovich signature of Brownian motion $t \mapsto (S(B_{0, t}))_{t \in [0, T]}$ is a geometric $p$-rough path almost surely for $p>2$. 
\end{lemma}

We are interested in the random signature $S\left(B_{[0, \tau_{\Omega }]}\right)$ of the Brownian path up to the first exit time $\tau_\Omega$. It is shown  in \cite{lyons2015expected} that $\pi^{I}\left(S\left( B_{%
\left[ 0,\tau _{\Omega }\right] }\right)\right)$ has finite expectation with respect to the Wiener measure for every index $I$. Thus, the tensor-valued function $S\left( B_{%
\left[ 0,\tau _{\Omega }\right] }\right)$ is integrable.
\begin{definition}
 We denote by $\Phi _{\Omega }(z)$ the expected signature of Brownian motion starting at $z$ and stopped upon the first exit time $\tau _{\Omega }$ from a domain $\Omega $. That is,
\begin{equation*}
\Phi _{\Omega }(z)=\mathbb{E}^{z}\left[S\left(B_{ \left[ 0,\tau_{\Omega  }\right] }\right)\right].
\end{equation*}
\end{definition}


One of the main results in \cite{lyons2015expected} states that $\Phi$ satisfies a system of nested PDEs:
\begin{theorem}
\label{PDE_theorem}\label{PDE_theorem_Phi}
Let $\Omega \subset \R^d$ be a bounded domain. Then $\Phi$ satisfies
\begin{equation}\label{ES_BM_PDE}
\Delta \Phi(z)=-\left(\sum_{i=1}^{d}e_{i}\otimes e_{i}\right)\otimes
\Phi (z)-2\sum_{i=1}^{d}\left( e_{i}\otimes \frac{\partial \Phi
}{\partial z^{i}}(z)\right) \quad\text{for each  }z\in \Omega ,
\end{equation}%
with the boundary condition that for every $z \in \Omega $,
\begin{equation}\label{bc, original}
\lim_{t \nearrow \tau_{\Omega }}\Phi(B_{t})=\mathbf{1}\qquad \text{a.s. in } P^{z}  ,
\end{equation}%
and the initial conditions
\begin{eqnarray}
\proj_{0}(\Phi_{\Omega}(z)) &=&1 \qquad\text{for each }z\in\overline{\Omega }, \label{ic 1} \\
\proj_{1}(\Phi_{\Omega}(z)) &=&0 \qquad\text{for each }z\in \overline{\Omega}\label{ic 2}.
\end{eqnarray}
\end{theorem}

\subsection{Hyperbolic development}

Let $\mathscr{A}$ be a normed algebra. Any linear operator $M: E\to\mathscr{A}$ can be extended to $T((E))$ \emph{by naturality}. Indeed, one first defines $\hat{M}$ on $E^{\otimes k}$ for each $k$ via
\begin{eqnarray*}
\hat{M}\left(e_{i_{1}} \otimes e_{i_{2}} \otimes \cdots \otimes e_{i_{k}}\right) := M(e_{i_{1}}) M(e_{i_{2}}) \cdots M(e_{i_{k}}), 
\end{eqnarray*}
and then extends to $T((E))$ by linearity. The right-hand side of the above identity is understood as a product in the algebra $\mathscr{A}$. By an abuse of notations, we always write $\hat{M} \equiv M$.

For any $\lambda >0$, the action of $\lambda M$ on the signature of a path $\gamma$ of bounded $1$-variation is given by $M(\lambda \gamma)$. Similarly, for $\lambda >0$, the action of $\lambda M$ on $\Phi_{\Omega}$ is given by
\begin{equation}\label{dev-new}
\mathcal{M}_{\lambda, \Omega}(z) := (\lambda M) \Phi_{\Omega}(z) \quad \text{ for each } z \in \Omega.
\end{equation}
This can be recast into
\begin{equation}\label{M}
\mathcal{M}_{\lambda, \Omega}(z) =  \sum_{n=0}^\infty \lambda^n M\left[ \proj_n(\Phi_\Omega(z))\right]
\end{equation}
at least for small enough $|\lambda|$. See Lyons--Ni \cite{ni2012expected}, Theorem 3.6.

The hyperbolic development (\emph{cf.} Boedihardjo--Diehl--Mezzarobba--Ni \cite{boedihardjo2021expected}; Hambly--Lyons \cite{hl}) refers to following particular choice of $M$, denoted as $H: T((E)) \to \mathfrak{gl}(d+1;\R)$ ---
\begin{eqnarray*}
&& H(c) =  c I_{d+1} \qquad \text{ for } c \in \R,\\
&&
H (x,y) := \begin{bmatrix}
0&0&\cdots &0 & z^1\\
0&0&\cdots &0 & z^2\\
\vdots & \vdots &\ddots &\vdots& \vdots \\
0&0&\cdots &0 & z^d\\
z^1&z^2&\cdots &z^{d}&0
\end{bmatrix}\qquad \text{ for } z = \left(z^1,\ldots,z^d\right) \in E,
\end{eqnarray*}
and $H$ is extended to $T((E))$ by naturality as above. 

Lyons--Xu \cite{lyons2017hyperbolic} characterises the space of the hyperbolic development of bounded 1-variation $\mathbb{R}^{d}$ path and provides an explicit expression for the hyperbolic development in terms of the signature. Specifically, for $\gamma: [0, 1] \rightarrow \mathbb{R}^{d}$ of bounded $1$-variation and $h_{\lambda}(\gamma):= H\left(S\left(\lambda\gamma\right)\right) (0, \cdots, 0, 1)^{\top}$, it is shown that $h_{\lambda}(\gamma)$ is confined within the $d$-dimensional hyperboloid
\begin{align*}
    \mathbb{H}_{d}:=\left\{ x\in\R^{d+1}:\, \sum_{j=1}^{d+1} \left(x^j\right)^2 - \left(x^{d+1}\right)^2=-1,\, x^{d+1}>0 \right\}.
\end{align*}
Denote the set of ``squared words'' by
\begin{eqnarray*}
\mathcal{E}_{2n}=\Big\{w = (i_{1}, i_{1}, i_{2}, i_{2}, \cdots, i_{n}, i_{n}) :\,i_1, i_2, \cdots, i_n \in \{1, 2, \cdots, d\}\Big\}
\end{eqnarray*}
and, for $k \in\{1, 2, \cdots, d\}$, put 
\begin{eqnarray*}
\mathcal{E}_{2n}^{(k)}=\Big\{w = (i_{1}, i_{1}, i_{2}, i_{2}, \cdots, i_{n}, i_{n}, k):\,i_1, i_2, \cdots, i_n \in \{1, 2, \cdots, d\}\Big\}.
\end{eqnarray*}

\begin{lemma}\label{lemma_hyperbolic_dev}
Let $\gamma: [0, T] \rightarrow \mathbb{R}^{d}$ be a path of bounded $1$-variation, and set $$h_{\lambda}(\gamma)\equiv \left(h_{\lambda}^{(1)}(\gamma), \cdots, h_{\lambda}^{(d+1)}(\gamma) \right)^{\top} := H\left(S\left(\lambda\gamma\right)\right).$$ We may express
\begin{eqnarray*}
&&h_{\lambda}^{(k)} (\lambda\gamma) = \sum_{n \geq 0} \sum_{ w \in \mathcal{E}_{2n}^{(k)}} \lambda^{2n+1} \proj_{w}(S(\gamma)),\\
&&h_{\lambda}^{(d+1)} (\lambda\gamma) = \sum_{n \geq 0} \sum_{ w \in \mathcal{E}_{2n}} \lambda^{2n}\proj_{w}(S(\gamma)).
\end{eqnarray*}
\end{lemma}

As in Chevyrev--Lyons \cite{cl}, the radius of convergence of the expected signature of a rough path contains crucial information about the law of the path.

\begin{definition}
The radius of convergence of a tensor $\varphi=(\varphi_n) \in T((E))$ is the radius of convergence of the series $\lambda\mapsto \sum_{n=0}^\infty \lambda^n \||\varphi_n\||$. Throughout this paper, $\||\bullet\||$ is taken to be the projective norm on $T((E))$. 
\end{definition}

The following characterisation for the finiteness of radius of convergence of the expected signature can be found in \cite{cl}.

\begin{proposition}\label{prop/def}
$\varphi \in T((E))$ has infinite radius of convergence if and only if it lies in the closure of $T((E))$ with respect to the coarsest topology for which the following holds: for any normed algebra $\mathscr{A}$ and any morphism $H \in {\rm Hom}(E;\mathscr{A})$, the natural extension $\hat{H}: T((E)) \to \mathscr{A}$ of $H$ (not relabelled in the sequel) is continuous in this topology.
\end{proposition}


\subsection{Further notations} A domain $\Omega \subset \R^d$ is a connected open set. The symbols $z^\dagger$, ${\rm Re}(z)$, and ${\rm Im}(z)$ denote respectively complex conjugate, real part, and imaginary part of $z \in \C$, $A^\top$ is the transpose of matrix $A$, and ${\bf id}$ is the identity map (the domain being clear from the context). We write $U \Subset V$ if the closure of $U$ is in $V$. Our notations in \S\ref{sec: proof of main thm} below are largely identical to  those in  \cite{boedihardjo2021expected}; \emph{e.g.}, the usage of symbols $\alpha$, $\zeta$, $M$, etc. For essentially analogous but slightly different symbols we shall distinguish by an overhead bar; for instance, we write $\bar{A}_\lambda$, $\bar{B}_\lambda$, and  $\bar{C}_\lambda \equiv \overline{\mathcal{F}_{\lambda,\e}}$ in contrast to $A_\lambda$, $B_\lambda$, and $C_\lambda={F}_{\lambda,\Omega}$ in \cite{boedihardjo2021expected}, respectively. Finally, $\mathfrak{gl}(D;\R) = \{ D \times D  \text{ real matrices}\}$ is equipped with the Hilbert--Schmidt norm as a normed algebra.

A bounded domain $\Omega \Subset \R^d$ is said to be of regularity $X$ ($=C^k$,  $C^{2,\alpha}$, etc.) if its boundary are locally graphs of functions of regularity $X$. More precisely, there exist open sets $\mathcal{O}_1,\ldots,\mathcal{O}_N \subset \R^d$ such that for each $j\in\{1,2,\ldots,N\}$, it holds that
\begin{itemize}
\item
$\p\Omega \subset \bigcup_{j=1}^N \mathcal{O}_j$;
    \item 
$\mathcal{O}_j \cap \p\Omega \neq \emptyset$;
    \item
there are an open set $V_j\subset \R^{d-1}$ and a function $\psi_j:V_j\to\R$ of regularity $X$ such that $\mathcal{O}_j \cap \p\Omega = \left\{(y,\psi_j(y)):\,y \in V_j \right\}$ and $\mathcal{O}_j \cap \Omega = \left\{ (y,z)\in V_j\times \R:\,z<\psi_j(y)\right\}$.
\end{itemize}

\section{PDE for the expected signature}\label{sec: lemmata}

Here we investigate the PDE~\eqref{ES_BM_PDE} for the expected signature. Our results in this section are valid for $E=\R^d$ for any $d\in \mathbb{Z}_{\geq 2}$. Also note that our regularity assumption on the domain is weaker than that in Lyons--Ni \cite{lyons2015expected}. See Remark~\ref{rem: elliptic} below.

\begin{proposition}\label{prop: schauder}
Let $\Omega \subset E$ be a bounded $C^{2,\alpha}$-domain; $\alpha > 0$. Then for each $n = 2,3,4,\ldots$, the $n^{\text{th}}$ term of $\Phi_{\Omega}$ satisfies the PDE
\begin{align}
  \Delta\left(\proj_{n}(\Phi_{\Omega}(z))\right)
=  -2\sum_{i=1}^{d}e_{i}\otimes\frac{\partial\proj_{n-1}(\Phi_{\Omega }(z))}{\partial z^{i}} -\left(\sum_{i=1}^{d}e_{i}\otimes e_{i}\right)\otimes\proj_{n-2}(\Phi_{\Omega }(z))\label{eq:LyonsNiPDE}
\end{align}
with the initial conditions~\eqref{ic 1}, \eqref{ic 2} and the boundary condition
\begin{equation}\label{bc, Holder domain}
    \Phi_{\Omega} = {\bf 1} \qquad\text{ on } \p \Omega. 
\end{equation}
The above PDE has a unique classical solution in the sense that
\begin{equation*}
    \proj_{n}(\Phi_{\Omega}) \in C^\infty\left(\Omega; E^{\otimes n}\right) \cap C^{2,\alpha}\left(\overline{\Omega}; E^{\otimes n}\right) \qquad \text{ for each } n=0,1,2,\ldots.
\end{equation*}
In addition, the radius of convergence of $\Phi$ is strictly positive. 
\end{proposition}


\begin{proof}

The proof is standard and shall only be sketched here. See  \cite{lyons2015expected}, \S 3 for details.

The boundary condition~\eqref{bc, Holder domain} follows from Eq.~\eqref{bc, original} and a continuity argument. 

Given the initial conditions~\eqref{ic 1} and \eqref{ic 2}, one may apply the standard Schauder theory (\cite{pde, gt}) inductively to show that  Eq.~\eqref{eq:LyonsNiPDE} has a  unique classical solution
\begin{equation*}
    \proj_{n}(\Phi_{\Omega}) \in C^\infty\left(\Omega; E^{\otimes n}\right) \cap C^{2,\alpha}\left(\overline{\Omega}; E^{\otimes n}\right).
\end{equation*}
As a remark, although the Schauder theory is not applicable for systems of elliptic PDEs in general, we can decouple the PDE~\eqref{eq:LyonsNiPDE} into scalar equations, thus there is no danger of utilising the Schauder theory here. More precisely, let $\{e_1,\ldots,e_d\}$ be the canonical basis for $E$; we can find scalarfields $f_{i_1\cdots i_n}:\Omega\to\R$ such that $$\Phi_n(z) = \sum_{i_1, \ldots i_n \in \{1,2,\ldots, d\}}\, f_{i_1\cdots i_n}(z) \, e_{i_1} \otimes \ldots \otimes e_{i_n}.$$ Thus Eq.~\eqref{eq:LyonsNiPDE} is equivalent to finitely many scalar Poisson equations of $f_{i_1\cdots i_n}$ for each $n$.

In addition, Schauder estimates together with a simple induction yield the geometric bound:
\begin{equation}
\label{geom bd}
    \|\proj_{n}(\Phi_{\Omega})\|_{C^0(\Omega)} \leq C^n,
\end{equation}
with $C$ depending only on  $\Omega$ and $d$, which implies that the radius of convergence of $\Phi$ is strictly positive. See  \cite{lyons2015expected}, Theorem~3.6 or the appendix in  \cite{boedihardjo2021expected}.    \end{proof}

\begin{remark}\label{rem: elliptic}
In \cite{lyons2015expected}, \S 3, a variant of Proposition~\ref{prop: schauder} is established for  strong solutions in Sobolev spaces $W^{m,2}(\Omega)$ for each $m \in \mathbb{N}$. Then the geometric bound~\eqref{geom bd} can be deduced from the  Sobolev--Morrey embedding, provided that $m > \lfloor n\slash 2 \rfloor +1$. But this requires, as in \cite{lyons2015expected},  that the domain $\Omega$ is $C^m$. Proposition~\ref{prop: schauder} indicates that such regularity assumption can be relaxed to $C^{2,\alpha}$.

In fact, thanks to the Kellogg theorem \cite{kellogg} one may further relax the regularity of $\Omega$ to $C^{2, {\rm Dini}}$, {\it i.e.}, $\Omega$ admits $C^2$-local charts with parametrisation maps being merely Dini continuous. In this case we have the classical solution in $C^\infty\left(\Omega; E^{\otimes n}\right) \cap C^{2}\left(\overline{\Omega}; E^{\otimes n}\right)$, while the geometric bound~\eqref{geom bd} remains valid.

\end{remark}

Recall the multi-index notation:
\begin{eqnarray*}
\left|I\right|=i_1+\ldots+i_d,
\qquad
D^Iu=\frac{\partial^{i_{1}+\ldots i_d}u}{\partial z_{1}^{i_{1}}\partial z_{2}^{i_{2}}\cdots\partial z_{d}^{i_{d}}},
\qquad
I=\left(i_{1},\ldots,i_d\right)\in\mathbb{N}^{d}.
\end{eqnarray*}
As an immediate consequence of Proposition~\ref{prop: schauder} and Remark~\ref{rem: elliptic}, we have
\begin{corollary}\label{cor: differentiability}
Let $\Omega \subset E$ be a bounded domain of class $C^{2,\alpha}$ (or even  $C^{2,{\rm Dini}}$). Let $M \in {\rm Hom}\big(\Omega;T((E))\big)$ be an algebra homomorphism with the projective norm on $T((E))$. Then there exists a constant $L>0$ depending only on $\Omega$ and the operator norm of $M$ such that, for all $\lambda \in [0,L[$, the series $z\mapsto\sum_{n=0}^{\infty}\lambda^{n}M\left[\proj_{n}\left(\Phi_{\Omega}(z)\right)\right]$ defines a $C^2$-function on $\Omega$. In addition, for any multi-index $I$ with  $|I|\leq 2$, it holds that
\begin{equation*}
  D^{I}\left(\sum_{n=0}^{\infty}\lambda^{n}M\left[\proj_{n}\left(\Phi_{\Omega}(z)\right)\right]\right)
=  \sum_{n=0}^{\infty}\lambda^{n}D^{I}\Big\{M\left[\proj_{n}\left(\Phi_{\Omega}(z)\right)\right]\Big\}.
\end{equation*}
\end{corollary}


Theorem \ref{PDE_theorem} shows that for a bounded domain, the expected signature of the stopped Brownian motion satisfies the characteristic PDE~\eqref{thm: pde}. Conversely, the classical solution ($C^2$ in the interior and $C^1$ up to the boundary) for the characteristic PDE must coincide with the expected signature of the stopped Brownian motion. 


\begin{theorem}\label{thm: pde}
Let $\Omega$ be a bounded domain in $E$. Suppose that $\xi: \overline{\Omega} \rightarrow T((E))$ is a classical solution to the following PDE, in the regularity class $C^2\big(\Omega; T((E))\big)\cap C^1\big(\overline{\Omega};T((E))\big)$:
\begin{equation*}
    \begin{cases}
    \Delta \proj_{n}(\xi) = -2 \sum_{i=1}^d e_i \otimes \frac{\p  \proj_{n-1}(\xi)}{\p z^i}- \sum_{i=1}^d e_i \otimes e_i \otimes \proj_{n-2}(\xi)\qquad \text{in } \Omega;\\
   \xi = {\bf 1}  \qquad \text{ on } \p\Omega;\\
    \proj_{0}(\xi) \equiv 1  \text{ and } \proj_{1}(\xi) \equiv 0\qquad \text{in } \Omega.
    \end{cases}
\end{equation*}
Then $\xi(z)$ is the expected signature of Brownian motion starting at $z$ up to the first exit time $\tau_{\Omega}$; \emph{i.e.}, $\xi = \Phi_{\Omega}$ on $\overline{\Omega}$.
\end{theorem}

Let us outline the main ideas. Following the martingale approach in \cite{ni2012expected} for the diffusion process up to a deterministic time, if we can construct a martingale ${N}_{t}$ based on $\xi$ such that: \begin{itemize} \item[(a)] $\lim_{t\rightarrow \infty}{N}_{t} = S(B_{0, \tau_{\Omega}})$ where the limit are taken in both the almost surely and $L^{1}$-sense (note that $S(B_{0, \tau_{\Omega}})$ is proven to be $L^{1}$-integrable); and that  \item[(b)] ${N}_{0} = \xi(z)$, \end{itemize} then we have $\mathbb{E}^{z}\left[\underset{t\rightarrow \infty }{\lim}{N}_{t}\right] = \mathbb{E}^{z}[{N}_{0}] $ and hence $ \mathbb{E}^{z}[S(B_{0, \tau_{\Omega}})]= \xi(z)$ by Doob's martingale convergence theorem.

The next question is how to construct the martingale ${N}_{t}$.  For this purpose, note that $\hat{N}_{t}:= \mathbb{E}^{z}[S(B_{0, \tau_{\Omega}}) \vert \mathcal{F}_{t\wedge \tau_{\Omega}}]$ is a martingale (by Tower's property) and satisfies the above condition. Thus, in light of Chen's identity, $\hat{N}_{t}$ can be rewritten as $ S(B_{0, t \wedge \tau_{\Omega}}) \otimes \Phi(B_{t\wedge \tau_{\Omega}})$. This motivates us to consider $\hat{N}_{t}:=S(B_{0, t \wedge \tau_{\Omega}}) \otimes \xi(B_{t\wedge \tau_{\Omega}})$. Once we prove that $\hat{N}_{t}$ is a martingale, we may conclude that the PDE solution $\xi$ coincides with the expected signature $\Phi_{\Omega}$.

\begin{proof}
Since $\xi_{0} = \proj_0(\Phi_{\Omega}) \equiv  1$ and $\xi_{1} = \proj_{1}(\Phi_{\Omega}) \equiv 0$, we only need to prove that for any $n \geq 2$, one has $\xi_{n} = \proj_{n}(\Phi_{\Omega})$. To this end, set $N_{t} := S(B_{0, t}) \otimes \xi(B_{t})$. First of all, let us show that $N_{t}$ is a martingale, or equivalently, that $\proj_{n}(N_{t})$ is a martingale for every $n \geq 2$.

Indeed, note that
\begin{align}\label{sde_martingle}
\dd\pi_{n}(N_{t}) = \dd\proj_{n}\Big(S(B_{0, t}) \otimes \xi(B_{t})\Big) 
= \dd \left\{\sum_{i = 0}^{n}\proj_{i}\big(S(B_{0, t})\big) \otimes \xi_{n-i}(B_{t})\right\}
\end{align}
and, by the definition of the signature of Brownian motion, it holds that
\begin{align*}
\dd\proj_{n}(S(B_{0, t})) &= \proj_{n-1}(S(B_{0, t}))\otimes \,\dd B_{t} \nonumber\\
&=\proj_{n-1}(S(B_{0, t}))\otimes \cdot \,\dd B_{t} + \frac{1}{2}\proj_{n-2}(S(B_{0, t})) \otimes \sum_{j=1}^{d}  e_{j}\otimes e_{j}\, \dd t,
\end{align*}
where $\cdot \dd B_{t}$ is understood in the It\^{o} sense and $n \geq 2$.

On the other hand, we infer from the It\^o formulae that
\begin{eqnarray*}
\dd\xi_{n}(B_{t}) = \sum_{i = 1}^{d}\partial_{z_{i}}\xi_{n}(B_{t})\,\cdot \dd B^{i}_{t} + \frac{1}{2}\Delta \xi_{n}(B_{t})\,\dd t.
\end{eqnarray*}
Thus, the SDE of $\pi_{n}({N}_{t})$ in Eq.~\eqref{sde_martingle} can be rewritten as 
\begin{align*}
&\dd \Big(\proj_{i}\big(S(B_{0, t})\big) \otimes \xi_{n-i}(B_{t}) \Big) \\
&\quad = \dd \proj_{i}\big(S(B_{0, t})\big) \otimes \xi_{n-i}(B_{t}) +  \proj_{i}\big(S(B_{0, t})\big) \otimes \dd\xi_{n-i}(B_{t}) + \dd\Big\langle \proj_{i}\big(S(B_{0, t})\big),  \xi_{n-i}(B_{t})\Big\rangle \\
&\quad= \mathbf{1}_{i\geq 2}\frac{1}{2}\proj_{i-2}\big(S(B_{0, t})\big) \otimes \sum_{j=1}^{d}e_{j}\otimes e_{j} \,\dd t \otimes \xi_{n-i}(B_{t}) \\
&\qquad+ \proj_{i}\big(S(B_{t})\big)\otimes \frac{1}{2}\Delta \xi_{n-i}(B_{t})\,\dd t+ \sum_{j =1}^{d} \mathbf{1}_{i\geq 1}\proj_{i-1}\big(S(B_{t})\big)\otimes e_{j}\otimes \partial_{z_{j}}\xi_{n-i}(B_{t})+ Q^{n, i}_{t}\cdot \dd B_{t},
\end{align*}
where $i \leq n$ and
\begin{align*}
Q^{n, i}_{t} &= \proj_{i}(S(B_{0, t}))\otimes \sum_{j = 1}^{d}\partial_{z_{j}}\xi_{n-i}(B_{t})\,\cdot \dd B_{t} + \mathbf{1}_{i\geq 1}\proj_{i-1}(S(B_{0, t}))\otimes \cdot \dd B_{t} \otimes \xi_{n-i}(B_{t}).
\end{align*}

Summing the above equation over $i \in \{0, \ldots, n\}$, we get
\begin{eqnarray*}
&&\dd \proj_{n}(N_{t}) = \sum_{i = 0}^{n-2}\proj_{i}\big(S(B_{t})\big)\otimes \mathcal{R}_{n - i-2}(B_{t}) \,\dd t +\sum_{i = 0}^{n}Q^{n, i}_{t}\cdot \dd B_{t},
\end{eqnarray*}
where 
\begin{eqnarray*}
\mathcal{R}_{i}(z) = \frac{1}{2} \sum_{j=1}^{d}e_{j}\otimes e_{j} \otimes \xi_{i}(z) + \sum_{j=1}^{d} e_{j}\otimes \partial_{z_{j}}\xi_{i+1}(z) + \frac{1}{2}\Delta \xi_{i+2}(z). 
\end{eqnarray*}

By the fact that $\xi$ is a solution to the given PDE, the drift term equals zero. Therefore, $N_{t}$ is a local martingale. It then implies that $Y_{t}:= N_{t \wedge \tau_{\Omega}}$ is a local martingale.

It remains to show that $Y_{t}$ is a martingale. Indeed, for any $n$ and $i \in \{0, 1, \cdots, n\}$, and for all $t \in \mathbb{R}^{+}$, we have  $\mathbb{E}\left[\left| Q^{n,i}_{t \wedge \tau_{\Omega}}\right|^{2}\right] < \infty$. This is because for every $i \in \{0, 1, 2, \cdots\}$, it holds that
\begin{eqnarray*}
&& \mathbb{E}^{z}\left[\vert \proj_{i}(S(B_{0, t \wedge \tau_{\Omega}}))\vert^{2}\right] < \infty, \\
&& \mathbb{E}^{z}\left[\Big\vert \xi_{i}(B_{t \wedge \tau_{\Omega}})\Big\vert^{2} + \sum_{j=1}^{d} \Big\vert\partial_{z_{j}} \xi_{i}(B_{t \wedge \tau_{\Omega}})\Big\vert^{2}\right] < \infty.
\end{eqnarray*}
Then $\mathbb{E}^{z}[\vert Y_{t}\vert] < \infty$ by H\"{o}lder's inequality, and hence $Y_{t}$ is a martingale. Moreover, $\mathbb{E}^{z}[\vert Y_{t}\vert]$ is uniformly bounded, thus $\mathbb{E}^{z}[Y_{0}] = \mathbb{E}^{z}[Y_{t}]$ for all $t >0$. In addition,  $Y_{\infty}= S(B_{0, \tau_{\Omega}}) \in L^{1}$, so
\begin{eqnarray*}
\mathbb{E}^{z}[Y_{0}] = \lim_{t \rightarrow \infty} \mathbb{E}^{z}[Y_{t}] = \mathbb{E}^{z}\left[\lim_{t \rightarrow \infty }Y_{t}\right].
\end{eqnarray*}
We conclude by noting that $\mathbb{E}^{z}[Y_{0}] = \xi(z) $ and  $\mathbb{E}^{z}[\lim_{t \rightarrow \infty }Y_{t}] = \mathbb{E}^{z}[S(B)_{0, \tau_{\Omega}}]$. \end{proof}

\section{Symmetrization of PDEs and Hyperbolic development}\label{sec: symm}
To overcome the difficulty that general domains are not rotationally invariant, we integrate the development $\mathcal{M}_{\lambda,\Omega}$ (see Eq.~\eqref{dev-new}) over all the rotated domains centred at $x \in \Omega$. The resulting object is referred to as the ``domain-averaging development'' of expected signature in the sequel. 

In this section, we first explore  general properties of domain-averaging developments for any $M \in {\rm Hom}\big(T((E)), \mathfrak{gl}(d+1;\R)\big)$, and then specialise to the hyperbolic development $H=M$. 


\subsection{Domain-averaging development of the expected signature }
We start with defining several operations on domains. The translation of domain $\Omega$ by $h$ is $\Omega+h:= \{z+h:\, z \in \Omega \}$. For $R \in SO(d)$, write $\Omega_{R}:=R(\Omega)$, the domain obtained by rotating $\Omega$ with respect to the centre $x = 0$ (see Figure~\ref{fig:rotated_domain} for a 2-dimensional illustration). 

Fix an arbitrary $x \in \Omega$. As $\Omega$ is open, there is $\e \in ]0,1[$ such that $\overline{\B(x,\e)} \subset \Omega$. We first note that the expected signature is invariant under translations, namely that $\Phi_{\Omega}(x):=\Phi_{\Omega+h}(x+h)$ whenever $x \in\Omega$. Without loss of generality we may assume $x= 0 \in \Omega$, since otherwise one simply takes  $\tilde{\Omega} = \Omega - x$.

Recall that 
\begin{equation}\label{M}
\mathcal{M}_{\lambda, \Omega}(z) := \lambda M (\Phi_{\Omega} (z)) =  \sum_{n=0}^\infty \lambda^n M \proj_n[\Phi_\Omega(z)]
\end{equation}
for sufficiently small $\lambda$ (\cite{ni2012expected}, Theorem~3.6). Any $z \in \B(0, \e)$ lies in $\bigcap_{S\in SO(d)} \Omega_{S}$, so  $\mathcal{M}_{\lambda, \Omega_R}(z)$ is well defined for each $R \in SO(d)$.  

\begin{figure}
\centering
\includegraphics[scale = 0.7]{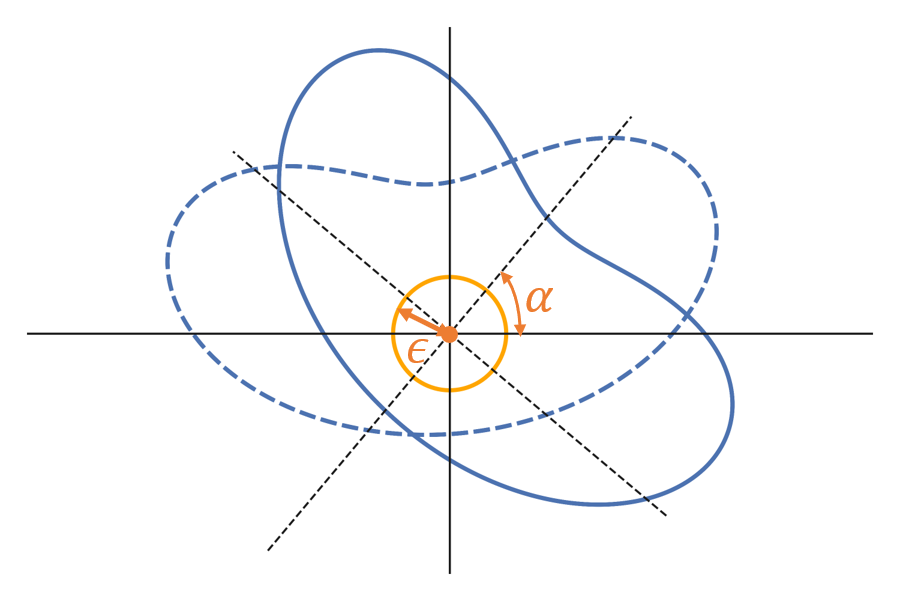}
\caption{Rotation of the 2-dimensional domain $\Omega$ (whose boundary is the solid blue curve) by degree $\alpha = \pi\slash 4$ with respect to the centre marked in orange. The boundary of the rotated domain $R_{\alpha}(\Omega)$ is the dashed blue curve, where $R_\alpha \in SO(2)$ is the rotation matrix by $\alpha$.}
\label{fig:rotated_domain}
\end{figure}

\begin{definition}[Domain-averaging development of expected signature]\label{def: domain avg}
Let $\Omega\Subset \R^d$ be a domain containing $\overline{\B(0,\e)}$. Set
\begin{equation}\label{M avg}
\overline{\mathcal{M}_{\lambda,\e,\Omega}}(z) = \int_{SO(d)}\mathcal{M}_{\lambda, \Omega_{R}}(z) \,\dd\chi(R) \in \mathfrak{gl}(d+1;\R)\qquad \text{ for each } z \in \overline{\B(0,\e)},
\end{equation}
where $\chi$ is the normalised Haar measure on $SO(d)$, and integration of matrices in $\mathfrak{gl}(d+1;\R)$ is understood in the entry-wise sense. For ease of notations, we write $\overline{\mathcal{M}_{\lambda,\e}} := \overline{\mathcal{M}_{\lambda,\e,\Omega}}$.
\end{definition}

\begin{lemma}[Rotational ``symmetry'' of $\mathcal{M}_{\lambda,\Omega}$]
For any $z \in \B(0, \e)\subset \Omega$ and $R \in SO(d)$,
\begin{eqnarray}
\left(R \oplus  {\bf id}\right) \mathcal{M}_{\lambda, \Omega}(z)\left(R^{\top} \oplus {\bf id}\right) =  \mathcal{M}_{\lambda, \Omega_{R}}(Rz). \label{B}
\end{eqnarray}
\end{lemma}
\begin{proof}
This extends the proof of Corollary~4 for the unit disc in \cite{boedihardjo2021expected}, where ${\bf id}$ is the identity map on $\R^1$ and $R^\top$ the transpose of $R$. Since the Brownian motion $B^{Rz}$
starting at $Rz$ has the same distribution as the rotated Brownian motion $R (B^{z})$ (with $B^z$ starting from $z$), for any $z \in \bigcap_{R\in SO(d)}R(\Omega)$ it holds that 
\begin{eqnarray*}
\mathbb{E}^{z}\left[S\left(RB\right)_{\tau_{\Omega}}\right]= \mathbb{E}^{Rz}\left[S(B)_{\tau_{R(\Omega)}}\right].\label{A}
\end{eqnarray*}
Applying $\lambda M$ to both sides of the above equation, we have
\begin{eqnarray*}
\lambda M \left(\mathbb{E}^{z}\left[S\left(RB\right)_{\tau_{\Omega}}\right]\right) = \lambda M\left( \mathbb{E}^{R(z)}\left[S(B)_{\tau_{R(\Omega)}}\right]\right).
\end{eqnarray*}
By Lemma 3 in \cite{boedihardjo2021expected} (which clearly holds in arbitrary dimension), it equals
\begin{eqnarray*}
\lambda M \left(S\left(RB\right)_{\tau_{\Omega}}\right)= \left(R \oplus  {\bf id}\right)\left\{ \lambda M\left[S\left(B\right)_{\tau_{\Omega_R}}\right]\right\} \left(R^{\top} \oplus {\bf id}\right).
\end{eqnarray*}
As $R$ is deterministic, it is furthermore equal to 
\begin{eqnarray*}
\left(R \oplus  {\bf id}\right) \Big\{(\lambda M)\mathbb{E}^{z}\left[S(B)_{\tau_{\Omega}}\right]\Big\}\left(R^{\top} \oplus {\bf id}\right) =  (\lambda M)\mathbb{E}^{Rz}\left[S(B)_{\tau_{\Omega_{R}}}\right]. 
\end{eqnarray*}
This concludes the proof.
\end{proof}

\begin{lemma}\label{Lemma_M_bar}
Write $r=|z|$ for $z \in \B(0,\e) \subset \Omega$. Then for any $R \in SO(d)$,
\begin{equation}\label{C}
(R \oplus {\bf id})   \overline{{\mathcal{M}}_{\lambda,\e}}(r) \left(R^{\top} \oplus {\bf id}\right) = \overline{\mathcal{M}_{\lambda,\e}}(z).
\end{equation}
\end{lemma}

\begin{proof}
Let $z = R (r, \underbrace{0, \cdots, 0}_{d-1 \text{ zeros}})^{\top}\in \B(0,\e) \subset \bigcap_{O \in SO(d)}\Omega_O \subset \Omega$. We deduce from Eq.~\eqref{B} that 
\begin{equation*}
(R \oplus {\bf id}) \mathcal{M}_{\lambda, \Omega_{S}}(r)\left(R^{\top} \oplus {\bf id}\right) = \mathcal{M}_{\lambda, \Omega_{R\circ S}}\left(R(r, 0, \cdots, 0)^{\top}\right) = \mathcal{M}_{\lambda, \Omega_{R\circ S}}(z)
\end{equation*}
for any $S \in SO(d)$. We can now conclude by integrating both sides of  this equation against the Haar measure $\mu$ on $SO(d)$.   
\end{proof}

As in \cite{boedihardjo2021expected} let us consider $
    F_{\lambda, \Omega} := M_{\lambda, \Omega}(z) \cdot [\underbrace{0,\cdots, 0}_{d-1\text{ zeros}},1]^\top$, and similarly
\begin{equation}\label{F and M}
\overline{\mathcal{F}_{\lambda,\e}}(z) := \overline{\mathcal{M}_{\lambda,\e}}(z)\cdot [\underbrace{0,\cdots, 0}_{d-1\text{ zeros}},1]^\top.
\end{equation}
The following can be deduced directly from Lemma~\ref{Lemma_M_bar}.
\begin{lemma}[Separation of variables]\label{Lemma_F_bar}
For any $z \in \B(0,\e) \subset \Omega$ and $R \in SO(d)$, it holds that
\begin{equation}\label{eqn_F_bar}
\overline{\mathcal{F}_{\lambda,\e}}(z) =
(R \oplus {\bf id})   \overline{{\mathcal{F}}_{\lambda,\e}}\left( (r, \underbrace{0, \cdots, 0}_{d-1 \text{ zeros}})^{\top}\right).
\end{equation}
\end{lemma}

In addition, $\overline{\mathcal{F}_{\lambda,\e}}$ and $\mathcal{F}_{\lambda, \Omega}$ are related in the following manner:
\begin{lemma}
For $\B(0, \e) \subset \Omega$, let $\overline{C_{\lambda, \e}}(r) = [\underbrace{0,\cdots, 0}_{d-1\text{ zeros}}, 1]  \overline{\mathcal{F}_{\lambda,\e}}(r)$ and $C_{\lambda}(z) = [\underbrace{0,\cdots, 0}_{d-1\text{ zeros}}, 1] \mathcal{F}_{\lambda,\Omega}(z)$. Then $
C_{\lambda}(0) = \overline{C_{\lambda, \e}}(0)$.
\end{lemma}
\begin{proof}
By Lemma~\ref{Lemma_M_bar} we have, for each $R \in SO(d)$, that
\begin{eqnarray*}
(R \oplus {\bf id})  \mathcal{M}_{\lambda,\Omega}(0) \left(R^{\top} \oplus {\bf id}\right) = \mathcal{M}_{\lambda,\Omega_R}(0).\end{eqnarray*}
The $(d+1, d+1)^{\text{th}}$ entry of both sides of the above equation is $
 C_{\lambda,\Omega}(0) =  C_{\lambda,\Omega_R}(0)$. So, integrating both sides over the angular variable gives us $C_{\lambda,\Omega}(0) =  \overline{C_{\lambda,\e}}(0).$  \end{proof}

\begin{lemma}[Preservation of finite radius of convergence]\label{lemma_FRC}
For $\B(0, \e) \subset \Omega$, if  $\overline{\mathcal{M}_{\lambda, \e}}(0)$ has finite radius of convergence, then so does $\mathcal{M}_{\lambda, \Omega}(0)$.
\end{lemma}

\begin{proof}
By Lemma~\ref{Lemma_M_bar} we have $(R \oplus {\bf id})  \mathcal{M}_{\lambda,\Omega}(0) \left(R^{\top} \oplus {\bf id}\right) = \mathcal{M}_{\lambda,\Omega_R}(0)$, which is equivalent to
\begin{eqnarray*}
 \mathcal{M}_{\lambda,\Omega}(0)  = (R^{\top} \oplus {\bf id})  \mathcal{M}_{\lambda,\Omega_{R}}(0)\left(R \oplus {\bf id}\right).
\end{eqnarray*}
The assertion follows immediately from the definition of radius of convergence. \end{proof}

\subsection{Symmetrization of PDEs}\label{subsec:key}
By Corollary~\ref{cor: differentiability} (compare also with \cite{ni2012expected}),  the map $z \mapsto \Phi_{\Omega}(z)$ is twice differentiable on  $\Omega$. This leads to

\begin{theorem}[PDE for $\overline{\mathcal{F}_{\lambda,\e}}$]\label{thm: pde for average, new}
Assume $\B(0,\e) \subset \Omega$. There exists a $\lambda^{*} >0$ such that for any $\lambda \in \C$ with $|\lambda| < \lambda^{*}$, $\mathcal{F}_{\lambda, \Omega}$ is twice continuously differentiable on $\B(0,\e)$. Moreover, it satisfies
\begin{equation}
\Delta \overline{\mathcal{F}_{\lambda, \e}}(z)=-2\lambda\sum_{i=1}^{d}Me_{i}\frac{\partial \overline{\mathcal{F}_{\lambda, \e}}}{\partial z_{i}}(z)-\lambda^{2}\left(\sum_{i=1}^{d}\left(Me_{i}\right)^{2}\right)\overline{\mathcal{F}_{\lambda, \e}}(z) \qquad\text{for any $z \in \B(0,\e)$}.\label{eq:Lemma2PDE} 
\end{equation}
\end{theorem}
\begin{proof}
Consider
\begin{eqnarray}\label{lambda star}
\lambda^{*} = \sup\left\{|\lambda| :\,\sum_{n = 0}^{\infty} \lambda^{n} \left\vert\left\vert \proj_{n}(\Phi_{\Omega}(z))\right\vert\right\vert < \infty \text{ for all }  z \in \B(0, \e)\right\}.
\end{eqnarray}
As remarked above, $z \mapsto \overline{\mathcal{F}_{\lambda, \e}(z)}$ is $C^2$ on $\B(0, \e)$. The definition of $\overline{\mathcal{F}_{\lambda, \e}}$ and Eq.~\eqref{M avg} give us
\begin{equation}\label{F avg}
\overline{\mathcal{F}_{\lambda,\e}}(z) =  \int_{SO(d)}\mathcal{F}_{\lambda, \Omega_{R}}(z) \,\dd\chi(R).
\end{equation}

We shall use the PDE for $\mathcal{F}_{\lambda, \Omega_{R}}$ to derive that for $\overline{\mathcal{F}_{\lambda,\e}}$. Indeed, for any $R \in SO(d)$ we have
\begin{equation*}
\Delta \mathcal{F}_{\lambda, \Omega_{R}}(z)=-2\lambda\sum_{i=1}^{d}Me_{i}\frac{\partial \mathcal{F}_{\lambda, \Omega_{R}}}{\partial z_{i}}(z)-\lambda^{2}\left(\sum_{i=1}^{d}\left(Me_{i}\right)^{2}\right)\mathcal{F}_{\lambda, \Omega_{R}}(z).  
\end{equation*}
As $\mathcal{F}_{\lambda,\Omega_R}$ is $C^2$, we can differentiate under integral signs to get
\begin{align}\label{eqn_F_alpha}
&\int_{SO(d)} \Delta \mathcal{F}_{\lambda, \Omega_{R}}(z) \,\dd\chi(R)\nonumber\\
&\qquad =-2\lambda\sum_{i=1}^{d}Me_{i}\int_{SO(d)} \frac{\partial \mathcal{F}_{\lambda, \Omega_{R}}}{\partial z_{i}}(z)\,\dd\chi(R) -\lambda^{2}\left(\sum_{i=1}^{d}\left(Me_{i}\right)^{2}\right)\overline{ \mathcal{F}_{\lambda, \e}(z)}.  
\end{align}   \end{proof}

\subsection{Hyperbolic development} Recall that the hyperbolic development (Hambly--Lyons \cite{hl}; Boedihardjo--Diehl--Mezzarobba--Ni \cite{boedihardjo2021expected}) is the morphism $H: T((E)) \to \mathfrak{gl}(d+1;\R)$ ---
\begin{eqnarray*}
&&H (x,y) := \begin{bmatrix}
0&0&\cdots &0 & z^1\\
0&0&\cdots &0 & z^2\\
\vdots & \vdots &\ddots &\vdots& \vdots \\
0&0&\cdots &0 & z^d\\
z^1&z^2&\cdots &z^{d}&0
\end{bmatrix}\qquad \text{ for } z = (z^1,\ldots,z^d) \in E,\\
&&H(v_1 \otimes\cdots\otimes v_n) := H(v_1) \cdot \ldots \cdot H(v_n)\quad \text{ for any $n\in\mathbb{N}$ and $v_1,\ldots,v_n \in E$}.
\end{eqnarray*}
Let $\overline{\mathcal{H}_{\lambda,\e}}$ denote the last column of domain-averaging development of expected stopped Brownian motion $\overline{\mathcal{F}_{\lambda, \e}}$ corresponding to the hyperbolic development $H$.

\begin{theorem}[ODE system of $\overline{\mathcal{H}_{\lambda,\e}}$]\label{Thm_F_bar}
Write, in the case $d=2$,
\begin{align*}
    \overline{\mathcal{H}_{\lambda,\e}}(r) =: \left[\bar{A}_\lambda(r),\bar{B}_{\lambda}(r), \bar{C}_{\lambda}(r)\right]^\top.
\end{align*}
When $|\lambda|<\lambda^\star$ in Eq.~\eqref{lambda star}, the following ODE system holds:
\begin{eqnarray}
&&r^{2}\bar{A}_{\lambda}^{''} + r\bar{A}_{\lambda}^{'}(r) -\bar{A}_{\lambda}(r) + \lambda^{2}r^{2}\bar{A}_{\lambda}(r) + 2\lambda r^{2}\bar{C}_{\lambda}^{'}(r) = 0,\label{A eq}\\
&&r^2\bar{B}_{\lambda}^{''}(r) + r\bar{B}_{\lambda}'(r) + (\lambda^2r^2-1)\bar{B}_{\lambda}(r)= 0,\label{B eq}\\
&&\bar{C}^{'}_{\lambda}(r) + r\bar{C}^{''}_{\lambda}(r) + 2\lambda^{2}r\bar{C}_{\lambda}(r) + 2\lambda r \bar{A}^{'}_{\lambda}(r) + 2\lambda \bar{A}_{\lambda}(r)= 0.\label{C eq}
\end{eqnarray}
To ensure that $\overline{\mathcal{H}_{\lambda,\e}}$ is $C^2$ at the origin $x=0$, one needs to impose the conditions
\begin{equation}
\label{boundary conditions, ABC}
\bar{A}_{\lambda}(0) = 0,\quad\bar{B}_{\lambda}(0) = 0,\quad\bar{C}^{'}_{\lambda}(0) = 0.
\end{equation}
\end{theorem}
\begin{proof}
As in \cite{boedihardjo2021expected}, a straightforward computation using Eq.~\eqref{C} and Lemma \ref{Lemma_M_bar} yields that 
\begin{equation}\label{F sep var}
\overline{\mathcal{H}_{\lambda,\e}}\left(r e^{i\theta}\right) = (R_{\theta} \oplus {\bf id})   \overline{\mathcal{H}_{\lambda,\e}}(r),
\end{equation}
where $R_\theta=\begin{bmatrix}
\cos\theta&-\sin\theta\\
\sin\theta & \cos\theta
\end{bmatrix}\in SO(2)$. The PDE for $\overline{\mathcal{H}_{\lambda,\e}}$ is the same as that for $\mathcal{H}_{\lambda,\B(0, 1)}$. 
Adapting almost verbatim the arguments in \cite{boedihardjo2021expected}, Lemma~6, we conclude that the ODE for $\overline{\mathcal{H}_{\lambda,\e}}$ is the same as that for $\mathcal{H}_{\lambda,\B(0, 1)}$ given in Eqs.~\eqref{A eq},  \eqref{B eq},  and \eqref{C eq}. \end{proof}

In contrast to the  case that $\Omega$ is the unit disc, roughly speaking, the PDE for $\overline{\mathcal{H}_{\lambda, \e}}$ misses ``half of the boundary conditions'', namely the conditions on $\partial \B(0, \e)$. To address this problem, we make use of geometric properties of the hyperbolic development.


\begin{lemma}\label{lem: hyperbolic dev}
Let $\B(0, \e) \subset \Omega$ and let $\lambda^\ast$ be as in Eq.~\eqref{lambda star}. For every $\lambda \in [0, \lambda^{*}[$ and $r \in [0, \e]$, set $\overline{\mathcal{H}_{\lambda, \e}}(r) := \left[h_{\lambda}^{(1)}(r), \cdots, h_{\lambda}^{(d+1)}(r)\right]^\top$. Then  $h^{(d+1)}_{\lambda}(r) \geq 1$ for every $r \in [0,\e]$.
\end{lemma}

\begin{proof}
As the hyperbolic development is a normed algebra morphism, for any $\lambda \in [0, \lambda^{*}[$ with $\lambda^{*}$ defined in Eq.~ \eqref{eq:Lemma2PDE},
\begin{align*}
||(\lambda H)(\Phi_{\Omega}(z))|| &= \left\|\sum_{n\geq 0} \lambda^n H(\proj_n(\Phi_{\Omega}(z))\right\|\\
&\leq \sum_{n \leq 0} \lambda^{n} \|H\|^n\|\proj(\Phi_{\Omega}(z))\|\\
&\leq \sum_{n \leq 0} \lambda^{n}\|\proj(\Phi_{\Omega}(z))\|<\infty,
\end{align*}
where the norm of $H$ is $1$. 
By the dominated convergence theorem, we can interchange the expectation and $\lambda H$ to obtain that
\begin{eqnarray*}
\lambda H (\Phi_{\Omega}(z)) = \mathbb{E}^{z}\left[\lambda H \left(S(B)_{0, \tau_{\Omega}}\right)\right].
\end{eqnarray*} 

By Lemma~\ref{Lemma_BM_GRP}, the Stratonovich signature of the Brownian motion $\left(S(B)_{0, t}\right)_{t\in [0, T]}$ is a geometric rough path for any finite $T> 0$. That is, there exists a sequence of bounded 1-variation paths $\left(B^{m}_{[0,T]}\right)_{m =1}^{\infty}$ whose signatures  converge almost surely to $\left(S(B)_{0, t}\right)_{t \in [0, T]}$ in the $p$-variation distance for any $p >2$. For instance, we can choose $B^{m}_{[0, T ]}$ as the dyadic piecewise linear approximation of Brownian motion with mesh size $2^{-m}$ up to time $T$.

It then holds that 
\begin{eqnarray*}
\lim_{m \rightarrow \infty} \lambda H \left(S\left(B^{m}_{[0, T]}\right)\right) \overset{a.s}{=} \lambda H\left(S\left(B_{[0, T]}\right)\right).
\end{eqnarray*}
As $\mathbb{E}^{z}[\tau_{\Omega}]$ is almost surely finite, we can send $T$ and $m$ to $+\infty$ to get
\begin{eqnarray*}
\lim_{T \rightarrow \infty }\lim_{m \rightarrow \infty} \lambda H \left(S\left(B^{m}_{\left[0, T\wedge \tau_\Omega\right]}\right)\right) \overset{a.s}{=} \lambda H\left(S\left(B_{\left[0, \tau_\Omega\right]}\right)\right).
\end{eqnarray*}

Note that for any $T>0$ and $m \in \mathbb{N}$, the boundness of $1$-variation of $B^{m}_{[0, T]}$ in the almost sure sense ensures that $\lambda H\left(S\left(B^{m}_{[0, T]}\right)\right)$ is finite almost surly.
Lyons--Xu \cite{lyons2017hyperbolic} showed that any bounded 1-variation path $\gamma$ satisfies $
H\left(S(\gamma)\right)\cdot  [0,\cdots, 0,1]^\top \in \mathbb{H}_d$, with 
\begin{eqnarray*}\mathbb{H}_d := \left\{\left(x^{(1)},\cdots, x^{(d)},x^{(d+1)}\right)\in\R^{d+1}: \sum_{i = 1}^{d}\left(x^{(i)}\right)^2 - \left(x^{(d+1)}\right)^2=-1,\, x^{(d+1)}>0\right\}.
\end{eqnarray*}
That is, the last co-ordinate of the hyperbolic development of the signature of $B^{m}_{[0, T \wedge \tau_{\Omega}]} \in \mathbb{N}$ is no less than $1$ almost surely for all $m$ and $T>0$. Thus, the expectation of the limit of $\left(\lambda H(S(B^{m})_{[0, T \wedge \tau_{\Omega}}])^{(d+1)}\right)_{m, T}$, as $m,T \nearrow \infty$, is greater than or equal to $1$.

Therefore, we can conclude the desired result:
\begin{align*}
    h^{(d+1)}_{\lambda}(r) =  \left( \mathbb{E}^{z}\left[\lambda H\left(S\left(B_{\left[0, \tau_{\Omega}\right]}\right)\right)\right]\left(0, \cdots, 0, 1\right)^{\top} \right)^{(d+1)}\geq 1.
\end{align*}
\end{proof}

\begin{corollary}\label{cor: new}
Let $\mathbf{B}(0,\e)$, $\lambda^\ast$, and $\overline{\mathcal{H}_{\lambda, \e}}(r)$ be the same as in Lemma~\ref{lem: hyperbolic dev} above, where $r \in [0, \e]$. Then $h^{(d+1)}_{\lambda}(r) \geq \left|h^{(1)}_{\lambda}(r)\right|$ for every $\lambda \in [0, \lambda^{*}[$ and $r \in [0,\e]$.
\end{corollary}

\begin{proof} 
By \cite{lyons2017hyperbolic} again, $H\left(S(\gamma)\right)\cdot  [0,\cdots, 0,1]^\top := \left(H^{(1)}, \cdots, H^{(d+1)}\right)^{\top}\in \mathbb{H}_d$ and $H^{(d+1)} > 0$ for any bounded 1-variation path $\gamma$. In particular, $\left(H^{(d+1)}\right)^2 = 1+\sum_{i = 1}^{d} \left(H^{(i)}\right)^{2}$. Thus $H^{(d+1)} \geq \left|H^{(1)}\right|$. Let $B^{m}_{0, T\wedge \tau}$ denote the dyadic approximation of Brownian motion as in the proof of Lemma~\ref{lem: hyperbolic dev}, which has bounded $1$-variation almost surely. Moreover, for $\lambda \in [0, \lambda^{*}]$,
$$
\lim_{T \rightarrow \infty}\lim_{m \rightarrow \infty}\lambda H (S(B^m)_{0, T\wedge \tau})) \overset{a.s.}{=}\lambda H (S(B)_{0, \tau})).
$$ 
Since for any $m \in \mathbb{N}$ and $T >0$, it holds almost surely that
$$
\lambda H (S(B^m)_{0, T\wedge \tau}))^{(d+1)} -\left|\lambda H (S(B^m)_{0, T\wedge \tau}))^{(1)}\right| \geq 0,
$$
we may infer by sending $m,T\nearrow \infty$ that $
h^{(d+1)}_{\lambda}(r) - \left|h^{(1)}_{\lambda}(r)\right| \geq 0.$   \end{proof}

\section{Proof of Theorem~\ref{thm: main}: $d=2$}
\label{sec: proof of main thm}

The domain-averaging hyperbolic development $\overline{\mathcal{H}_{\lambda, \e}}$ introduced in earlier sections overcomes the issue of lack of rotational invariance on general domains. Combining $\overline{\mathcal{H}_{\lambda, \e}}$ with the techniques developed in \cite{boedihardjo2021expected,cl} and utilising properties of the Bessel functions $J_\nu$ and $Y_\nu$ (see \cite{as,lozier2003nist}),  we arrive at the following proof of Theorem~\ref{thm: main} in the 2-dimensional case.



\begin{proof}[Proof of Theorem~\ref{thm: main}, $d=2$]

Assume for contradiction that the expected signature $\Phi\equiv\Phi_{\Omega}$ had infinite radius of convergence. We divide our arguments into six steps below.

\smallskip
\noindent
{\bf Step 1.}  Recall the PDE~\eqref{eq:LyonsNiPDE} for $\Phi$ (reproduced below, see Theorem~\ref{thm: pde}):
\begin{equation}\label{PDE, general domain}
 \begin{cases}
    \Delta \left( \proj_{n}\Phi\right) = -2 \sum_{i=1}^2 e_i \otimes \frac{\p ( \proj_{n-1}\Phi)}{\p z^i}- \sum_{i=1}^2 e_i \otimes e_i \otimes\proj_{n-2}\Phi\quad\text{ in }\Omega;\\
\Phi = {\bf 1}  \qquad \text{ on } \p\Omega;\\
    \Phi_0 \equiv 1  \text{ and } \Phi_{1} \equiv 0\quad\text{ in }\Omega.
    \end{cases}
\end{equation}



By standard elliptic PDE theory and induction, for each $n=0,1,2,\ldots$ there exists a unique solution for Eq.~\eqref{PDE, general domain} that is real-analytic in the interior. Indeed, one may begin with the same arguments as for Proposition~\ref{prop: schauder} to reduce Eq.~\eqref{PDE, general domain} to scalar elliptic PDEs, which renders  applicable the standard elliptic regularity theory. For any subdomain  $\widetilde{\Omega} \Subset \Omega$ and  $n=2,3,4\ldots$, a simple induction and the Lax--Milgram lemma (\cite{gt}, \S 5.8) yield the existence of $\Phi_n \in W^{1,2}\left(\widetilde{\Omega}; E^{\otimes n}\right)$, such that the boundary value is attained in the trace sense as usual, and that for any test function $\psi \in C^\infty_c\left(\widetilde{\Omega}\right)$ we have
\begin{align*}
    \int_\Omega \na\psi \cdot \na \Phi_n\,\dd x = \int_\Omega \left\{ 2 \sum_{i=1}^d e_i \otimes \frac{\p  \Phi_{n-1}}{\p z^i} + \sum_{i=1}^d e_i \otimes e_i \otimes\Phi_{n-2}\right\}\psi \,\dd x.
\end{align*}
Then, by the interior regularity Theorem~8.10 in \cite{gt} we can bootstrap the regularity to $\Phi_n \in C^{\infty}\left(\widetilde{\Omega}; E^{\otimes n}\right) = \bigcap_{k=0}^\infty W^{k,2}\left(\widetilde{\Omega}; E^{\otimes n}\right)$. The real-analyticity of $\Phi_n$ follows since, by induction,  $\Delta \Phi_n $ equals a real-analytic function in the interior of $\Omega$; \emph{cf.} \cite{john}. 

On the other hand, when $\Omega$ is a $C^{2,\alpha}$-domain,  straightforward adaptations of the arguments in Lyons--Ni \cite{lyons2015expected}, \S\S 3.5.1--3.5.2 (by changing Sobolev spaces to suitable H\"{o}lder spaces) together with the $C^{2,\alpha}$-boundary regularity theory for elliptic PDE (\cite{gt}, Theorem~6.19)  yield the geometric decay bound for the expected signature:
\begin{align*}
\|\proj_n \Phi\|_{C^2(\Omega)} \leq (C_0)^n.
\end{align*}
Here $C_0$ depends only on the geometry of $\Omega$ and the H\"{o}lder index $\alpha$.


\smallskip
\noindent
{\bf Step 2.} Fix an arbitrary $x \in \Omega$. There is $\e \in ]0,1[$ such that $\overline{\B(x,\e)} \subset \Omega$. As  $\Phi_{\Omega}(z) = \Phi_{\Omega-x}(z-x)$ for any $z\in\Omega$, we can take $x = 0$ without loss of generality.

Recall from Theorem~\ref{Thm_F_bar} that 
$\overline{\mathcal{F}_{\lambda,\Omega}}(r) \equiv \left[\bar{A}_\lambda(r),\bar{B}_{\lambda}(r), \bar{C}_{\lambda}(r)\right]^\top$ satisfies
\begin{eqnarray}
&&r^{2}\bar{A}_{\lambda}^{''}(r) + r\bar{A}_{\lambda}^{'}(r) -\bar{A}_{\lambda}(r) + \lambda^{2}r^{2}\bar{A}_{\lambda}(r) + 2\lambda r^{2}\bar{C}_{\lambda}^{'}(r) = 0,\label{A eq2}\\
&&r^2\bar{B}_{\lambda}^{''}(r) + r\bar{B}_{\lambda}'(r) + (\lambda^2r^2-1)\bar{B}_{\lambda}(r)= 0,\label{B eq2}\\
&&\bar{C}^{'}_{\lambda}(r) + r\bar{C}^{''}_{\lambda}(r) + 2\lambda^{2}r\bar{C}_{\lambda}(r) + 2\lambda r \bar{A}^{'}_{\lambda}(r) + 2\lambda \bar{A}_{\lambda}(r)= 0\label{C eq2}
\end{eqnarray}
for $r \in [0,\e[$, with the boundary conditions
\begin{align}\label{boundary conditions, ABC, reproduced}
\bar{A}_{\lambda}(0) = 0,\quad\bar{B}_{\lambda}(0) = 0,\quad \text{ and }\quad  \bar{C}^{'}_{\lambda}(0) = 0.
\end{align}  

\smallskip
\noindent
{\bf Step 3.} From now on, let us focus on Eqs.~\eqref{A eq2} and \eqref{C eq2} for $\bar{A}_\lambda$ and $\bar{C}_\lambda$; our blowup quantity will  be  $\bar{C}_\lambda(0)$. Introduce the ``missing boundary conditions'':
\begin{equation}\label{missing bc}
\bar{A}_\lambda(\e) =: \bb \quad\text{ and } \quad \bar{C}_\lambda(\e) =:\mathfrak{c}.     
\end{equation}
Here $\bb$ and $\mathfrak{c}$ depend in general on $\lambda$ and $\e$.

An adaptation of Boedihardjo--Diehl--Mezzarobba--Ni \cite{boedihardjo2021expected}, \S 6 allows us to explicitly solve the system~\eqref{A eq2} -- \eqref{missing bc}. Indeed, for $r \in [0,\e[$ we  consider the ansatz
\begin{align}\label{ansatz}
\bar{A}_\lambda(r) =\alpha  \mathscr{C}_1(\lambda\zeta r)\quad \text{and}\quad \bar{C}_\lambda(r) = \mathscr{C}_0(\lambda\zeta r),
\end{align}
where $\alpha,\zeta \in \C$ are complex parameters to be specified and $\mathscr{C}_\nu$ are cylinder functions ($\nu=0,1$); \emph{i.e.}, linear combinations of $J_\nu$ and $Y_\nu$, the order-$\nu$ Bessel function of the first and the second kind, respectively. The identities
\begin{align*}
&\mathscr{C}_0'=\mathscr{C}_1,\\
&s\mathscr{C}_1'(s) + \mathscr{C}_1(s) = s\mathscr{C}_0(s)
\end{align*} 
together with the defining equations for cylinder functions
\begin{align}\label{std bessel ODE}
s^2\mathscr{C}_\nu''(s)+s\mathscr{C}_\nu'(s)+(s^2-\nu^2)\mathscr{C}_\nu(s)=0
\end{align}
show that for particular choices of $\zeta$ and $\alpha$, the ansatz~\eqref{ansatz} indeed solves Eqs.~\eqref{A eq} and \eqref{C eq}. Here $\zeta$ is any one of the four complex roots of $p(z)=z^4+z^2+2$, and $\alpha=\zeta+\zeta^3/2.$

Moreover, as $Y_0(s)\sim \mathcal{O}(\log s)$ and $Y_1(s)\sim \mathcal{O}(s^{-1})$ for $s \searrow 0$ (\cite{lozier2003nist}, \S 10.7(i)), in order for $\bar{A}_\lambda$ and $\bar{C}_\lambda$ to be bounded at the origin $r=0$, the cylinder functions in the ansatz~\eqref{ansatz} cannot contain $Y_0$ and $Y_1$. Then we are left with $J_0$ and $J_1$ only, and the corresponding ansatz automatically satisfies the boundary condition~\eqref{boundary conditions, ABC} at $r=0$.

Note also that the $C^2$-solution for the PDE system~\eqref{A eq2} -- \eqref{boundary conditions, ABC, reproduced} is unique.

To summarise, we have further reduced the ansatz~\eqref{ansatz} to
\begin{equation}\label{ansatz'}
\begin{cases}
\bar{A}_\lambda(r) = u\mm(r) + v \mm^\dagger(r),\\
\bar{C}_\lambda(r) = u\nn(r) + v\nn^\dagger(r), 
\end{cases}
\end{equation}
where $u,v$ are complex coefficients and $\mm$, $\nn: [0,\e[ \to \C$ are given by
\begin{equation}\label{m,n}
\begin{cases}
\mm(r):=\alpha J_1(\lambda\zeta r),\\
\nn(r):=J_0(\lambda\zeta r).
\end{cases}
\end{equation}

\begin{remark}
It should be emphasised that one needs $\lambda \in \R$ in the above arguments, which guarantees that $\mm^\dagger(r)=\alpha^\dagger J_1\left(\lambda \zeta^\dagger r\right)$ and $\nn^\dagger = J_0\left(\lambda\zeta^\dagger r\right)$.
\end{remark}

\smallskip
\noindent
{\bf Step 4.} The coefficients $u,v$ can be  easily solved from the boundary condition~\eqref{missing bc}. Indeed, Eqs.~\eqref{ansatz'} and \eqref{missing bc} imply that $
\begin{bmatrix}
\mm(\e)&\mm(\e)^\dagger\\
\nn(\e)&\nn(\e)^\dagger
\end{bmatrix}\begin{bmatrix}
u\\
v
\end{bmatrix}=\begin{bmatrix}
\bb\\
\mathfrak{c}
\end{bmatrix}$,
hence
\begin{align*}
\begin{bmatrix}
u\\
v
\end{bmatrix} = \frac{\begin{bmatrix}
\nn(\e)^\dagger&-\mm(\e)^\dagger\\
-\nn(\e)&\mm(\e)
\end{bmatrix}\begin{bmatrix}
\bb\\
\mathfrak{c}
\end{bmatrix}}{\mm(\e) \nn(\e)^\dagger-\nn(\e)\mm(\e)^\dagger}.
\end{align*} Setting $r=0$ in Eq.~\eqref{ansatz'}, we get
\begin{align}\label{C0, blowup}
\bar{C}_\lambda(0)&=\nn(0) u + \nn^\dagger(0)v\nonumber\\
&= \frac{\left[\nn^\dagger(\e)-\nn(\e)\right]\cdot \bb + \left[\mm(\e)-\mm^\dagger(\e)\right]\cdot\mathfrak{c}}{\mm(\e) \nn(\e)^\dagger-\nn(\e)\mm(\e)^\dagger}\nonumber\\
&= \frac{{\rm Im}\left\{-\bb J_0(\lambda\zeta\e)+\mathfrak{c}\alpha J_1(\lambda\zeta\e)\right\}}{{\rm Im}\left\{ \alpha^\dagger \left[J_1(\lambda\zeta \e)\right]^\dagger J_0(\lambda\zeta\e) \right\}}.
\end{align}
Again, in the last line of Eq.~\eqref{C0, blowup} the boundary data $\mathfrak{c}$, $\bb$ depend on both $\lambda$ and $\e$.

As computed in \cite{boedihardjo2021expected}, \S 6, the denominator ${\rm Im}\left\{ \alpha^\dagger \left[J_1(\lambda\zeta \e)\right]^\dagger J_0(\lambda\zeta\e) \right\}$ has a root $\lambda_\star \in \left]\frac{2.5}{\e}, \frac{3}{\e}\right[$. On the other hand, for such $\lambda_\star$ we have
\begin{align*}
{\rm Im} \left\{\mathfrak{c} \alpha J_1(\lambda_\star \zeta \e)\right\} \leq  -1.3\mathfrak{c}.
\end{align*}
See the proof of Lemma~8 in 
\cite{boedihardjo2021expected}. 
But Lemma~\ref{lem: hyperbolic dev} in this paper yields $\mathfrak{c} \geq 1$, for the third component of the hyperbolic development must be confined to the upper branch of the hyperbola. So the numerator in the right-most term in Eq.~\eqref{C0, blowup} is non-zero. This implies that the radius of convergence of $\overline{\mathcal{H}_{\lambda,\e}}$ (the domain-averaged hyperbolic development of the expected signature) is no larger than $\lambda_\star$, \emph{provided that $\bb=0$}.

\smallskip
\noindent
{\bf Step 5.} 
It remains to consider the case $\bb \neq 0$. Then the numerator equals
\begin{align*}
{\rm Im}\left\{\mathfrak{c}\alpha J_1(\lambda_\star\zeta\e)-\bb J_0(\lambda_\star\zeta\e)\right\} =  \bb\, {\rm Im}\left\{ \frac{\mathfrak{c}}{\bb} \alpha J_1(\lambda_\star \zeta \e) - J_0(\lambda_\star\zeta \e) \right\}
\end{align*}
when evaluated at $\lambda_\star$, since both $\mathfrak{c},\bb\in\R$. Again, by the geometry of hyperbolic developments (Lemma~\ref{lem: hyperbolic dev} and Corollary~\ref{cor: new}) we have $|\mathfrak{e}|\equiv|\mathfrak{c}\slash\bb|\geq 1$.

To this end, we shall establish in Appendix, Lemma~\ref{lem: appendix} that
\begin{align}\label{N, mu, e, WTS}
\mathcal{N}(\mu;\mathfrak{e}) := {\rm Im}\left\{ \mathfrak{e} \alpha J_1(\mu\zeta) - J_0(\mu\zeta) \right\} \neq 0\quad\text{ for all } \mu \in [2.5,3],\,|\mathfrak{e}|\geq 1.
\end{align}

\smallskip
\noindent
{\bf Step 6.} As a consequence, there exists at least one $R_0 \in SO(d)$ such that $C_{\lambda, \Omega_{R_0}}(x)$,  \emph{i.e.}, the final component of the hyperbolic development for the stopped Brownian motion on the rotated domain $R_0$, blows up. Therefore, in view of Lemma~\ref{lemma_FRC} (in which $\mathcal{M}=\mathcal{H}$), for any  $x \in \Omega \Subset \R^2$ the expected signature $\Phi_\Omega$ must have finite radius of convergence at $x$.    \end{proof}

\section{Proof of Theorem~\ref{thm: main}: $3\leq d \leq 8$}\label{sec: high d}

In this final section, fix any $d=2,3,\ldots$ and write $\mathbf{D}\equiv\mathbf{D}^d$ for the $d$-dimensional unit ball. 

\subsection{The unit ball}  Let $\mathcal{H}_{\lambda, \mathbf{D}}(z)$ be the hyperbolic development of the expected signature of a  $d$-dimensional Brownian motion starting from $z$ and stopped upon the first exit time from 
$\mathbf{D}$ scaled by $\lambda$. Note that the domain-averaging hyperbolic development of the  expected signature coincides with that of the hyperbolic development; that is,  $
\overline{\mathcal{H}_{\lambda, 1, \mathbf{D}}}(z) = \mathcal{H}_{\lambda, \mathbf{D}}(z)$.

\begin{lemma}[Sparseness of $\mathcal{H}_{\lambda, \mathbf{D}}$]\label{lem: sparseness}
For each $\lambda >0$, the hyperbolic development $\mathcal{H}_{\lambda, \mathbf{D}}\left(\mathbf{r}\right) = \left (\mathcal{H}^{(1)}_{\lambda, \mathbf{D}}(\mathbf{r}), \cdots, \mathcal{H}^{(d+1)}_{\lambda, \mathbf{D}}(\mathbf{r})\right )^{\top}$ evaluated at the point ${\bf r}=(r,0,\cdots,0)^\top\in \mathbb{R}^{d}$ satisfies $\mathcal{H}^{(i)}_{\lambda, \mathbf{D}}(\mathbf{r})=0$ for $i \in \{2, 3, \cdots, d\}$ and $\mathcal{H}^{(d+1)}_{\lambda, \mathbf{D}}(\mathbf{r}) \geq 1$ almost surely, as long as it is well-defined for such $\lambda$.
\end{lemma}

\begin{proof}
The second statement follows immediately from the construction of the hyperbolic development, as $\mathcal{H}_{\lambda, \mathbf{D}}^{(d+1)}(\lambda B) \geq 1$ almost surely. 

To prove the first statement, we let $\{\gamma_{j,t}\}_{j \in \mathbb{N}}$ be a sequence of paths of bounded $1$-variations converging almost surely and in $C^{0,1/2-}$ to $B_t$. Consider the reflected paths
\begin{align*}
    \hat{\gamma}_{j,t} := \left(  \gamma^{(1)}_{j,t}, -\gamma^{(2)}_{j,t}, \cdots,-\gamma^{(d)}_{j,t} \right)^\top.
\end{align*}
Then $\{\hat{\gamma}_{j,t}\}$ converges in the same topology to the reflected Brownian motion 
\begin{equation}\label{reflected BM}
\hat{B}_t: = \left(B^{(1)}_{t}, - B^{(2)}_t, \cdots, -B^{(d)}_t\right)^\top.
\end{equation}
If $B_t$ starts at some point $\mathbf{r} = (r, 0, \cdots, 0)^\top$, we may further require that $\hat{B}_{0} = \gamma_{j,0} = \hat{\gamma}_{j,0}= \mathbf{r}$ for each $j$. Also
\begin{eqnarray*}
H\left(S\left(\lambda \gamma_ {j,\tau}\right)\right)^{(k)} = \sum_{n \geq 0} \sum_{ w \in \mathcal{E}_{2n}^{(k)}} \lambda^{2n+1} \proj_{w}\left(S( \hat{\gamma}_{j,\tau})\right).
\end{eqnarray*}

Taking the expectation, we get
\begin{eqnarray*}
\mathbb{E}^{\mathbf{r}}\left[ \sum_{n \geq 0} \sum_{ w \in \mathcal{E}_{2n}^{(k)}} \proj_{w}(S(\lambda \gamma_{j,\tau}))\right] = \mathbb{E}^{\mathbf{r}}\left[ \sum_{n \geq 0} \sum_{ w \in \mathcal{E}_{2n}^{(k)}} \proj_{w}\left(S(\lambda \hat{\gamma}_{j,\tau})\right)\right] .
\end{eqnarray*}
On the other hand, for every $w \in \mathcal{E}_{2n}^{(k)}$ and $k \in \{2, \cdots, d\}$,  Lemma~\ref{lemma_hyperbolic_dev} implies that
\begin{equation*}
    \proj_{w}(S(\lambda \gamma_{j,\tau})) = -\proj_{w}\left(S(\lambda \hat{\gamma}_{j,\tau})\right)
\end{equation*}
and that
\begin{equation*}
    \mathbb{E}^{\mathbf{r}}\left[ \sum_{n \geq 0} \sum_{ w \in \mathcal{E}_{2n}^{(k)}} \proj_{w}(S(\lambda \gamma_{j,\tau}))\right] = -\mathbb{E}^{\mathbf{r}}\left[ \sum_{n \geq 0} \sum_{ w \in \mathcal{E}_{2n}^{(k)}} \proj_{w}(S(\lambda (\hat{\gamma}_{j,\tau})))\right].
\end{equation*}
Therefore, for $k \in \{2, \cdots, d\}$,
\begin{eqnarray*}
\mathcal{H}^{(k)}_{\lambda, \mathbf{D}}(\mathbf{r})=\mathbb{E}^{\mathbf{r}}\left[ \sum_{n \geq 0} \sum_{ w \in \mathcal{E}_{2n}^{(k)}} \proj_{w}(S(\lambda \gamma_{j,\tau}))\right] = 0.
\end{eqnarray*} 
Now we can pass to the limit $j\to\infty$ to conclude the same for $B_\tau$.  \end{proof}


Theorem~\ref{thm: pde for average, new} remains valid in $d$-dimensional unit ball. It is straightforward to check that  $\mathcal{H}_{\lambda, \mathbf{D}}(z)$ satisfies the trivial boundary condition for $z \in \partial \mathbf{D}$.  Thus we get

\begin{lemma}\label{lem: PDE_H_d_ball}
There exists $\lambda^* >0$ such that for any $\lambda \in \C$ with $|\lambda| < \lambda^*$, $\mathcal{H}_{\lambda, \mathbf{D}}$ is $C^2$ on $\mathbf{D}$. Moreover, it satisfies
\begin{equation*}
\Delta \mathcal{H}_{\lambda, \mathbf{D}}(z)=-2\lambda\sum_{i=1}^{d}He_{i}\frac{\partial \mathcal{H}_{\lambda, \mathbf{D}}}{\partial z^{i}}(z)-\lambda^{2}\left(\sum_{i=1}^{d}\left(He_{i}\right)^{2}\right)\mathcal{H}_{\lambda, \mathbf{D}}(z)
\end{equation*}
for each $z \in \mathbf{D}$, and it is subject to the boundary condition $\mathcal{H}_{\lambda, \mathbf{D}} = [0,\ldots, 0, 1]^{\top}$ on $\partial \mathbf{D}$. This PDE boundary value problem has a unique solution.
\end{lemma}

The separation of variables and symmetrization arguments in \S\ref{sec: symm} reduce the PDE in Lemma~\ref{lem: PDE_H_d_ball} to an ODE system in the radial variable $r$ only. Moreover, in view of Lemma~\ref{lem: sparseness}, this  ODE system is trivial except for the first and last components of $\mathcal{H}_{\lambda, \mathbf{D}}$. Here and hereafter, $r=|z|$ is the radial co-ordinate on $\mathbf{D}$ and $\chi$ is the normalised Haar measure on $SO(d)$.

\begin{proposition}[ODE system for $\mathcal{H}_{\lambda, \mathbf{D}}$]\label{propn: ODE for hyperbolic dev in d-dim}
Write $h_{\lambda}^{(1)}:=\overline{\mathcal{H}^{(1)}_{\lambda, \mathbf{D}}}$ and $h_{\lambda}^{(d+1)}:=\overline{\mathcal{H}^{(d+1)}_{\lambda, \mathbf{D}}}$ where 
\begin{equation*}
    \overline{\mathcal{H}_{\lambda, \mathbf{D}}}(r) := \int_{SO(d)} \mathcal{H}_{\lambda, \mathbf{D}}\left(R\cdot\left[r,\underbrace{0,\ldots,0}_{d-1\,\text{zeros}}\right]^\top\right) \,\dd\chi(R).
\end{equation*}
Then for $d\in \{2,3,4,\ldots\}$ it holds that
  \begin{eqnarray}
    &&\left(h^{(1)}_{\lambda} \right)'' + \frac{d-1}{r} \left(h^{(1)}_{\lambda}\right)'- \frac{(d-1) h^{(1)}_{\lambda}}{r^2} = -2\lambda \left(h^{(d+1)}_{\lambda}\right)' - \lambda^{2}h^{(1)}_{\lambda},\label{ODE one, d-dim}\\
    && \left(h^{(d+1)}_{\lambda}\right)'' +\frac{d-1}{r}\left(h^{(d+1)}_{\lambda}\right)' = - 2 \lambda \left\{\frac{d-1}{r}h^{(1)}_{\lambda} +\left(h^{(1)}_{\lambda}\right)'\right\} -  d \lambda^{2}h^{(d+1)}_{\lambda},\label{ODE two, d-dim}\\
    && h^{(1)}_\lambda (1) = 0,\qquad h^{(d+1)}_\lambda(1)=1.
    \end{eqnarray}
\end{proposition}

\begin{proof} By Lemma~\ref{Lemma_F_bar} we have $
{\mathcal{H}_{\lambda, \mathbf{D}}}(z) = \left(R \oplus \mathbf{id} \right) \overline{\mathcal{H}_{\lambda, \mathbf{D}}}(r)$ for $z =R\cdot(r,0,\ldots,0)^\top$. Also,  Lemma~\ref{lem: sparseness} gives us $
\overline{\mathcal{H}_{\lambda, \mathbf{D}}}(z) \in \left\{ w \in \mathbb{R}^{d+1}:\, w^{(k)} = 0\text{ for } k \in \{2, \cdots, d\} \text{ and }  w^{(d+1)}\geq 1\right\}.$ Thus
\begin{align*}
    {\mathcal{H}_{\lambda, \mathbf{D}}}(z) =\begin{bmatrix}
    R \cdot \begin{bmatrix}
    h_\lambda^{(1)}(r)\\
    0\\
    \vdots\\
    0
    \end{bmatrix} \\
 h_\lambda^{(d+1)}(r)
    \end{bmatrix} = \begin{bmatrix}
\frac{z^{1}}{r}h_{ \lambda}^{(1)}(r) \\\frac{z^{2}}{r}h_{ \lambda}^{(1)}(r)\\ \vdots \\\frac{z^{d}}{r}h_{ \lambda}^{(1)}(r)\\h^{(d+1)}_{\lambda}(r) \end{bmatrix},
\end{align*}
since the unit sphere $\sph$ is invariant under $R$.

Consider now the PDE system in Lemma~\ref{lem: PDE_H_d_ball}:
\begin{equation}\label{new, PDE, C}
\Delta \mathcal{H}_{\lambda, \mathbf{D}}=-2\lambda\sum_{i=1}^{d}He_{i}\frac{\partial \mathcal{H}_{\lambda, \mathbf{D}}}{\partial z^{i}}-\lambda^{2}\left(\sum_{i=1}^{d}\left(He_{i}\right)^{2}\right)\mathcal{H}_{\lambda, \mathbf{D}}.
\end{equation}
The coefficients can be  easily computed: since $H(e_i)$ is the matrix whose only nonzero entries are the $(i, d+1)^{\text{th}}$ and $(d+1, i)^{\text{th}}$ ones (both equal to $1$), we find that
\begin{equation}\label{proj identity A}
    \left(\sum_{i=1}^{d}\left(He_{i}\right)^{2}\right) = {\rm diag}(1,\cdots,1,d).
\end{equation}
Moreover, the $\alpha^{\text{th}}$ component of  $\sum_{i=1}^{d}He_{i}\frac{\partial \mathcal{H}_{\lambda, \mathbf{D}}}{\partial z^{i}}$ equals $
      {\p \mathcal{H}_{\lambda, \mathbf{D}}^{(d+1)}}\slash{\p z^\alpha}$ if $\alpha\in \{1,2,\ldots,d\}$ and $
      \sum_{i=1}^d \p \mathcal{H}_{\lambda, \mathbf{D}}^{(i)}\slash \p z^i$ if  $\alpha=d+1$. In the latter case we further have
      \begin{align}\label{proj identity B}
    \sum_{i=1}^d \frac{\p \mathcal{H}_{\lambda, \mathbf{D}}^{(i)}}{\p z^i}(z) &=   \sum_{i=1}^d\left\{ \frac{1}{r}h_{\lambda}^{(1)}(r) - \frac{z^i z^i}{r^3}h_{\lambda}^{(1)}(r) + \frac{z^i}{r}\left(h_{\lambda}^{(1)}(r)\right)'(r) \frac{\p r}{\p z_i} \right\}\nonumber \\
   &= \frac{d-1}{r}h_{\lambda}^{(1)}(r)+ \left(h_{\lambda}^{(1)}\right)'(r). 
      \end{align}

On the other hand, $z^1/r=\cos\theta_1$ in the usual spherical co-ordinates $\omega=z/r=(\theta_1, \ldots, \theta_{d-1})$ on $\mathbf{S}^{d-1}$. Note also that $\Delta$ on the left-hand side of Eq.~\eqref{new, PDE, C} is the Euclidean Laplacian on $\R^{d+1}=\R^d\times \R$; when acting on the $\R^d$-component, in the spherical polar co-ordinates 
\begin{align}\label{identity, beltrami}
    \Delta = \p_{rr} + \frac{d-1}{r} \p_r  + \frac{1}{r^2}\Delta_{\mathbf{S}^{d-1}},  
\end{align}
where $\Delta_{\mathbf{S}^{d-1}}$ is the Laplace--Beltrami operator on the sphere $\mathbf{S}^{d-1}$.

To proceed, let us project Eq.~\eqref{new, PDE, C} onto the first and the last components. We shall prove the blowup of the resulting ODE system for $h_\lambda^{(1)}$ and $h_\lambda^{(d+1)}$.

For the projection onto the $(d+1)^{\text{th}}$ component, note that $\left[\proj_{d+1} \mathcal{H}_{\lambda,\mathbf{D}}\right](z)=h_\lambda^{(d+1)}(r)$  by Eq.~\eqref{new, B}, which is independent of $\omega \in \mathbf{S}^{d-1}$. So Identity~\eqref{identity, beltrami} leads to
\begin{align*}
    \Delta \proj_{d+1} \mathcal{H}_{\lambda,\mathbf{D}}(z) = \left(h_\lambda^{(d+1)}\right)''(r) + \frac{d-1}{r}\left(h_\lambda^{(d+1)}\right)'(r).
\end{align*}
The projection of the right-hand side of Eq.~\eqref{new, PDE, C} can be found via Eqs.~\eqref{proj identity A} and \eqref{proj identity B}. Thus we obtain Eq.~\eqref{ODE two, d-dim}.

It remains to prove Eq.~\eqref{ODE one, d-dim}. To this end, note that Eqs.~\eqref{new, PDE, C} -- \eqref{identity, beltrami} yield that
\begin{align*}
    \Delta\left(\frac{z^1}{r} h_\lambda^{(1)}\right) = -2\lambda \frac{\p h_\lambda^{(d+1)}}{\p z^1} - \lambda^2\frac{z^1}{r}  h_1.
\end{align*}
Writing in the spherical polar co-ordinates, we have $z^1\slash r = \cos\theta_1$, $\p h_\lambda^{(d+1)}\slash{\p z^1} = \cos\theta_1 \left( h_\lambda^{(d+1)} \right)'$, as well as 
\begin{align*}
     \Delta\left(\frac{z^1}{r} h_\lambda^{(1)}\right) = \cos\theta_1\left\{ \left(h_\lambda^{(1)}\right)''+\frac{d-1}{r}\left(h_\lambda^{(1)}\right)'\right\}  + \frac{h_\lambda^{(1)}}{r^2} \Delta_{\mathbb{S}^{d-1}} \cos\theta_1.
\end{align*}
In addition, the Laplace--Beltrami $\Delta_{\mathbf{S}^{d-1}}$ on $\mathbf{S}^{d-1}$ can be expressed in terms of the Laplace--Beltrami $\Delta_{\mathbf{S}^{d-2}}$ on $\mathbf{S}^{d-2}$ --- for $\omega = (\theta_1, \omega')$ where $\omega'=(\theta_2,\cdots,\theta_{d-1}) \in \mathbf{S}^{d-2}$, one has
\begin{align*}
    \Delta_{\mathbf{S}^{d-1}} = \left(\sin\theta_1\right)^{2-d} \frac{\p}{\p\theta_1} \left( \left(\sin\theta_1\right)^{d-2}\frac{\p}{\p \theta_1} \right) + \frac{1}{\sin^2 \theta_1} \Delta_{\mathbf{S}^{d-2}},
\end{align*}
from which it follows that $$\Delta_{\mathbf{S}^{d-1}} \cos\theta_1 = -(d-1)\cos\theta_1.$$ We thus obtain Eq.~\eqref{ODE one, d-dim}.

The boundary conditions $h^{(1)}_\lambda (1) = 0$, $h^{(d+1)}_\lambda(1)=1$ follow from $\Phi_{\mathbf{D}}(1)=1$, which holds since the stopping time is zero on the boundary $\p\mathbf{D}$.   \end{proof}

\begin{remark}
When $d=2$, Proposition~\ref{propn: ODE for hyperbolic dev in d-dim} agrees with Theorem~\ref{Thm_F_bar} ($\bar{A}_\lambda = h_\lambda^{(1)}$ and $\bar{C}_\lambda = h_\lambda^{(3)}$). 
\end{remark}

To proceed, let us first introduce a few further notations. Define the dimensional constants
\begin{align*}
\eta_\pm:= d+3 \pm \sqrt{d^2-10d+9}
\end{align*}
and
\begin{align*}
\beta_\pm:=\sqrt{2d-6 \pm 2 \sqrt{d^2-10d+9}}.
\end{align*}
In each dimension $d \geq 2$, the Wr\"{o}nskian determinant $\mathcal{W}$ is a function of $\lambda$:
\begin{align}\label{wronskian, arb dim}
\mathcal{W} \equiv \mathcal{W}[\lambda] :=  \det\begin{bmatrix} 
J_{\frac{d}{2}}\left(\frac{\lambda\beta_+}{2}\right)  & J_{\frac{d}{2}}\left(\frac{\lambda\beta_-}{2}\right)\\
\eta_+\beta_+ J_{\frac{d}{2}-1}\left(\frac{\lambda\beta_+}{2}\right) & \eta_-\beta_- J_{\frac{d}{2}-1}\left(\frac{\lambda\beta_-}{2}\right)
\end{bmatrix}.
\end{align}
We can express the symmetrised development $\overline{\mathcal{H}_{\lambda, \mathbf{D}}}(r)$ on the $d$-dimensional unit disc as follows.
 
\begin{theorem}\label{thm: d-disk, PDE sol}
Let $\mathbf{D} = \mathbf{D}^{d}$ with $d \in \mathbb{N}\setminus \{1,9\}$. Let $\lambda_\star$ denote the radius of convergence of $\overline{\mathcal{H}_{\lambda, \mathbf{D}}}$. Then for all $[0, \lambda_\star[$ we have
\begin{align}\label{new, B}
    \overline{\mathcal{H}_{\lambda, \mathbf{D}}}(z)  = \left[
\frac{z^{1}}{r}h_{ \lambda}^{(1)}(r),\frac{z^{2}}{r}h_{ \lambda}^{(1)}(r),\, \cdots ,\,\frac{z^{d}}{r}h_{ \lambda}^{(1)}(r),\, h^{(d+1)}_{\lambda}(r) \right]^\top,
\end{align}
where
\begin{align}
&h_{\lambda}^{(1)}(r) = r^{1-\frac{d}{2}  }\frac{8d}{{\mathcal{W}}} \left\{J_{\frac{d}{2}}\left(\frac{\lambda\beta_-}{2}\right) J_{\frac{d}{2}}\left( \frac{\lambda  \beta_+ r}{2}\right) - J_{\frac{d}{2}}\left(\frac{\lambda\beta_+}{2}\right) J_{\frac{d}{2}}\left( \frac{\lambda \beta_- r}{2}\right)\right\},\label{B, sol, arb dim}\\
&h_{\lambda}^{(d+1)}(r) = r^{1-\frac{d}{2} }\frac{1}{{\mathcal{W}}} \bigg\{-\beta_+\eta_+   J_{\frac{d}{2}}\left(\frac{\lambda\beta_-}{2}\right) J_{\frac{d}{2}-1} \left( \frac{\lambda \beta_+ r}{2}\right)\nonumber\\
&\qquad\qquad\qquad\qquad\qquad\qquad +\beta_-\eta_-  J_{\frac{d}{2}}\left(\frac{\lambda\beta_+}{2}\right) J_{\frac{d}{2}-1}\left( \frac{\lambda \beta_- r}{2}\right)\bigg\}.\label{D, sol, arb dim}
\end{align}
For $d \in \{2, 3, \cdots, 8\}$ one may further express
\begin{align}
& h_{\lambda}^{(1)} (r) =  8d\,r^{1-\frac{d}{2}}\,\left(\frac{{\rm Im}\left\{\left[J_{\frac{d}{2}}\left(\frac{\lambda\beta_+}{2}\right)\right]^\dagger J_{\frac{d}{2}}\left(\frac{\lambda\beta_+ r}{2}\right) \right\}}{{\rm Im}\left\{\eta_-\beta_-\left[J_{\frac{d}{2}-1}\left(\frac{\lambda\beta_+}{2}\right)\right]^\dagger J_{\frac{d}{2}}\left(\frac{\lambda\beta_+}{2}\right) \right\}}\right),\label{AC for blowup, arb dim 1}\\
& h_{\lambda}^{(d+1)} = r^{1-\frac{d}{2}}\,\left(\frac{{\rm Im}\left\{\eta_{-}\beta_{-}\left[J_{\frac{d}{2}-1}\left(\frac{\lambda\beta_+}{2}\right)\right]^\dagger J_{\frac{d}{2}}\left(\frac{\lambda\beta_+ r}{2}\right) \right\}}{{\rm Im}\left\{\eta_-\beta_-\left[J_{\frac{d}{2}-1}\left(\frac{\lambda\beta_+}{2}\right)\right]^\dagger J_{\frac{d}{2}}\left(\frac{\lambda\beta_+}{2}\right) \right\}}\right)\label{AC for blowup, arb dim 2}.
\end{align}
\end{theorem}

\begin{proof}
The ODE system~\eqref{ODE one, d-dim}, \eqref{ODE two, d-dim} for $h^{(1)}_{\lambda}$ and $h^{(d+1)}_{\lambda}$ can be solved, \emph{e.g.}, by $\textsc{Maple}^\text{\textregistered}$.    Denote $\eta_\pm:= d+3 \pm \sqrt{d^2-10d+9}$ and $\beta_\pm:=\sqrt{2d-6 \pm 2 \sqrt{d^2-10d+9}}$. For each fixed $\lambda$, $h^{(1)}_{\lambda}(r)$ and $h^{(d+1)}_{\lambda}(r)$ are linear combinations of Bessel functions of the first and second kinds:
\begin{eqnarray}
&&h^{(1)}_\lambda(r) = \left(C_{1} J_{\frac{d}{2}}\left( \frac{\lambda  \beta_+ r}{2}\right) +  C_{2}J_{\frac{d}{2}}\left( \frac{\lambda \beta_- r}{2}\right)+ C_{3} Y_{\frac{d}{2}}\left( \frac{\lambda  \beta_+ r}{2}\right) +  C_{4}Y_{\frac{d}{2}}\left( \frac{\lambda \beta_- r}{2}\right)\right) r^{-\frac{d}{2}+1},\label{new-a}\\
&&h^{(d+1)}_\lambda(r) = - \frac{1}{8d}\Bigg\{C_1 \eta_+ \beta_+  J_{\frac{d}{2}-1} \left( \frac{\lambda \beta_+ r}{2}\right) + C_2 \eta_-\beta_-  J_{\frac{d}{2}-1}\left( \frac{\lambda \beta_- r}{2}\right)\nonumber\\
&&\qquad\qquad\qquad\qquad+C_3\eta_+\beta_+  Y_{\frac{d}{2}-1} \left( \frac{\lambda \beta_+ r}{2}\right) + C_4 \eta_-\beta_-  Y_{\frac{d}{2}-1}\left( \frac{\lambda \beta_- r}{2}\right)\Bigg\}r^{-\frac{d}{2}+1}. \label{new-b}
\end{eqnarray}

In view of the boundary conditions for $h^{(1)}_\lambda$, $h^{(d+1)}_\lambda$, that $\lim_{r \searrow 0}J_\nu(r)=0$,  and that 
\begin{align*}
Y_\nu(r) \sim -\frac{1}{\pi} \G(\nu) \left(\frac{z}{2}\right)^{-\nu}\quad\text{ as } r \searrow 0
\end{align*}
for $\nu = d/2$ or $d/2-1$ (\cite{lozier2003nist}, \S\S 10.7.3, 10.7.4), we get $C_3=C_4=0$; \emph{i.e.}, the solutions $(B_\lambda, D_\lambda)$ contains no branches of $Y_\nu$. Therefore, we arrive at a simplified ODE system:
\begin{eqnarray}
&&h^{(1)}_\lambda(r) = \left(C_{1} J_{\frac{d}{2}}\left( \frac{\lambda  \beta_+ r}{2}\right) +  C_{2}J_{\frac{d}{2}}\left( \frac{\lambda \beta_- r}{2}\right)\right) r^{-\frac{d}{2}+1},\label{eqn:d_ball_A}\\
&&h^{(d+1)}_\lambda(r) = - \frac{1}{8d}\left(C_1 \eta_+ \beta_+  J_{\frac{d}{2}-1} \left( \frac{\lambda \beta_+ r}{2}\right) + C_2 \eta_-\beta_-  J_{\frac{d}{2}-1}\left( \frac{\lambda \beta_- r}{2}\right)\right)r^{-\frac{d}{2} +1} \label{eqn:d_ball_C}
\end{eqnarray}
with boundary conditions
\begin{equation}\label{bc, arb dim, B and D}
h^{(1)}_\lambda(1) = 0\qquad \text{ and }\qquad h^{(d+1)}_\lambda(1)=1.
\end{equation}

Solving from the algebraic system~\eqref{eqn:d_ball_A},  \eqref{eqn:d_ball_C}, and \eqref{bc, arb dim, B and D} the constants $C_1$ and $C_2$, we get
\begin{align*}
C_1 = 8d\,\frac{ J_{\frac{d}{2}}\left(\frac{\lambda\beta_-}{2}\right)}{\mathcal{W}}\quad\text{ and }\quad C_2 = -8d\, \frac{ J_{\frac{d}{2}}\left(\frac{\lambda\beta_+}{2}\right)}{\mathcal{W}},
\end{align*}
with the Wr\"{o}nskian determinant
\begin{align}\label{wronskian, arb dim}
\mathcal{W} \equiv \mathcal{W}[\lambda] :=  \det\begin{bmatrix} 
J_{\frac{d}{2}}\left(\frac{\lambda\beta_+}{2}\right)  & J_{\frac{d}{2}}\left(\frac{\lambda\beta_-}{2}\right)\\
\eta_+\beta_+ J_{\frac{d}{2}-1}\left(\frac{\lambda\beta_+}{2}\right) & \eta_-\beta_- J_{\frac{d}{2}-1}\left(\frac{\lambda\beta_-}{2}\right)
\end{bmatrix}
\end{align}
provided that it is non-vanishing. Thus we arrive at
\begin{eqnarray}
&&h^{(1)}_\lambda(r) = r^{1-\frac{d}{2}}\frac{8d}{{\mathcal{W}}} \left\{J_{\frac{d}{2}}\left(\frac{\lambda\beta_-}{2}\right) J_{\frac{d}{2}}\left( \frac{\lambda  \beta_+ r}{2}\right) - J_{\frac{d}{2}}\left(\frac{\lambda\beta_+}{2}\right) J_{\frac{d}{2}}\left( \frac{\lambda \beta_- r}{2}\right)\right\},\label{B, sol, arb dim}\\
&&h^{(d+1)}_\lambda(r) = \frac{r^{1-\frac{d}{2}}}{{\mathcal{W}}} \left\{-\beta_+\eta_+   J_{\frac{d}{2}}\left(\frac{\lambda\beta_-}{2}\right) J_{\frac{d}{2}-1} \left( \frac{\lambda \beta_+ r}{2}\right) +\beta_-\eta_-  J_{\frac{d}{2}}\left(\frac{\lambda\beta_+}{2}\right) J_{\frac{d}{2}-1}\left( \frac{\lambda \beta_- r}{2}\right)\right\}.\label{D, sol, arb dim}
\end{eqnarray}

For $3\leq d \leq 8$ we set
\begin{equation*}
\Upsilon=\Upsilon(d) := \sqrt{-(d-1)(d-9)}. 
\end{equation*}
Then $\eta_\pm = d+3 \pm i\Upsilon$ and $\beta_\pm = \sqrt{2} \sqrt{d-3\pm i \Upsilon}$. A direct computation further leads to
\begin{align}\label{Re and Im parts of beta pm}
&\left[{\rm Re}\left({\beta_\pm}\right)\right]^2 = \sqrt{(d-3)^2+\Upsilon^2} + (d-3) = 2\sqrt{d}+d-3,\\
&\left[{\rm Im}\left({\beta_\pm}\right)\right]^2 = \sqrt{(d-3)^2+\Upsilon^2} - (d-3) = 2\sqrt{d}-d+3.
\end{align}
In the above, ${\rm Re}\left({\beta_+}\right)$ and ${\rm Im}\left({\beta_+}\right)$ have the same sign, and ${\rm Re}\left({\beta_-}\right)$ and ${\rm Im}\left({\beta_-}\right)$ have opposite signs. In view of the explicit formulae~\eqref{wronskian, arb dim}, \eqref{B, sol, arb dim}, and  \eqref{D, sol, arb dim}, we may choose without loss of generality that ${\rm Re}\left({\beta_\pm}\right)>0$. That is,
\begin{align}
\beta_\pm &= \sqrt{2\sqrt{d}+d-3} \pm i \sqrt{2\sqrt{d}-d+3}\nonumber\\
&= \sqrt{\left(\sqrt{d}+3\right)\left(\sqrt{d}-1\right)} \pm i \sqrt{\left(\sqrt{d}+1\right)\left(3-\sqrt{d}\right)}.
\end{align}

With the above choice one can further simplify the Wr\"{o}nskian in Eq.~\eqref{wronskian, arb dim}.  Indeed, as $\eta_\pm$ and $\beta_\pm$ are complex conjugates and, by \cite{lozier2003nist}, \S 10.11.9, $J_\nu\left(z^\dagger\right) = \left[J_\nu(z)\right]^\dagger$, it holds that
\begin{equation}\label{wronskian, purely imaginary}
\mathcal{W} = 2i\,{\rm Im}\left\{\eta_-\beta_-\left[J_{\frac{d}{2}-1}\left(\frac{\lambda\beta_+}{2}\right)\right]^\dagger J_{\frac{d}{2}}\left(\frac{\lambda\beta_+}{2}\right) \right\}.
\end{equation}
Similarly,
\begin{align*}
J_{\frac{d}{2}}\left(\frac{\lambda\beta_-}{2}\right) J_{\frac{d}{2}}\left( \frac{\lambda  \beta_+ r}{2}\right) - J_{\frac{d}{2}}\left(\frac{\lambda\beta_+}{2}\right) J_{\frac{d}{2}}\left( \frac{\lambda \beta_- r}{2}\right) = 2i\,{\rm Im}\left\{\left[J_{\frac{d}{2}}\left(\frac{\lambda\beta_+}{2}\right)\right]^\dagger J_{\frac{d}{2}}\left(\frac{\lambda\beta_+ r}{2}\right) \right\}.
\end{align*}
Eq.~\eqref{B, sol, arb dim} now yields closed-form expressions for $h_{\lambda}^{(1)}$ and $h_{\lambda}^{(d+1)}$, namely Eqs.~\eqref{AC for blowup, arb dim 1} and \eqref{AC for blowup, arb dim 2}.   \end{proof}

\begin{remark}
\cite{boedihardjo2021expected}, Lemma~7 for 2D unit disc is a special case of the above theorem. One can verify that the constants $\alpha,\zeta$ in \cite{boedihardjo2021expected} are given by $\alpha = -\frac{16\beta_{-}\eta_{-}}{\left|\beta_{-}\eta_{-}\right|^2} $ and $\zeta = \frac{\beta_{+}}{2}$.
\end{remark}

\begin{remark}
For $d = 9$, $\mathcal{W}$ is zero and Eqs.~\eqref{ODE one, d-dim} and \eqref{ODE two, d-dim} admit a fundamental solution $\overline{\mathcal{H}_{\lambda, \mathbf{D}}}$ in stark contrast with the solutions for $d \neq 9$ given in Eqs.~\eqref{new-a} and \eqref{new-b}. See Appendix A, Lemma~\ref{lem:d=9case}.
\end{remark}

The finiteness of  radius of convergence of the expected signature on $\mathbf{D}=\mathbf{D}^d$, $3\leq d \leq 8$ is deduced from the ODE system in Proposition~\ref{propn: ODE for hyperbolic dev in d-dim}.

\begin{theorem}\label{thm: d-dim, finite ROC}
The radius of convergence of the expected signature of the stopped Brownian motion  on $\mathbf{D}=\mathbf{D}^d$ is finite for $d \in \{2, \ldots,8\}$.
\end{theorem}

\begin{proof}

By Theorem \ref{thm: d-disk, PDE sol}, for $d \in \{2, \ldots, 8\}$ we have
\begin{eqnarray*}
h_{\lambda}^{(d+1)}(r) &=& r^{1-\frac{d}{2}}\,\left(\frac{{\rm Im}\left\{\eta_{-}\beta_{-}\left[J_{\frac{d}{2}-1}\left(\frac{\lambda\beta_+}{2}\right)\right]^\dagger J_{\frac{d}{2}}\left(\frac{\lambda\beta_+ r}{2}\right) \right\}}{{\rm Im}\left\{\eta_-\beta_-\left[J_{\frac{d}{2}-1}\left(\frac{\lambda\beta_+}{2}\right)\right]^\dagger J_{\frac{d}{2}}\left(\frac{\lambda\beta_+}{2}\right) \right\}}\right).  
\end{eqnarray*}
We shall prove that there exists a $\lambda_\star >0$ such that $h_{\lambda}^{(d+1)}(0)$ blows up as $\lambda \to \lambda_\star$.

Let us now analyse the zeros of the denominator
\begin{align}\label{XI}
\Theta(\lambda) := {\rm Im}\left\{\eta_-\beta_-\left[J_{\frac{d}{2}-1}\left(\frac{\lambda\beta_+}{2}\right)\right]^\dagger J_{\frac{d}{2}}\left(\frac{\lambda\beta_+}{2}\right) \right\}. 
\end{align}
Consider the Taylor expansion for $J_\nu(z)$; $\nu \in\left\{ \frac{d}{2}-1,\frac{d}{2}\right\}$. As in the appendix we write
\begin{align*}
J_{\nu}(z) = \mathring{J}_\nu^n(z) + \mathscr{R}_\nu^n(z),
\end{align*}
where
\begin{align*}
\mathring{J}_\nu^n(z)  := \left(\frac{z}{2}\right)^\nu \sum_{k=0}^{n-1} (-1)^k \frac{\left(\frac{z^2}{4}\right)^k}{k! \G(k+\nu+1)},
\end{align*}
and the remainder $\mathscr{R}_\nu^n(z)$ consists of summation from $k=n$ to $\infty$. Similarly, denote
\begin{align*}
\mathring{\Theta}^n(\lambda):={\rm Im}\left\{\eta_-\beta_-\left[\maino\left(\frac{\lambda\beta_+}{2}\right)\right]^\dagger \maine\left(\frac{\lambda\beta_+}{2}\right) \right\}.
\end{align*} The error $\left|\Theta(\lambda)-\mathring{\Theta}^n(\lambda)\right|$ can be estimated via Lemma~\ref{lem: bessel} and the triangle inequality:
\begin{align*}
&\left|\Theta(\lambda)-\mathring{\Theta}^n(\lambda)\right|\\
&\quad \leq \left|\eta_-\beta_-\right|\Bigg\{\left|\remo\left(\frac{\lambda\beta_-}{2}\right)\right| \left|\maine\left( \frac{\lambda\beta_+}{2} \right)\right| + \left|
 \remo\left(\frac{\lambda\beta_-}{2}\right)\right| \left|\reme\left( \frac{\lambda\beta_+}{2} \right)\right|\\
 &\qquad + \left|\maino\left(\frac{\lambda\beta_-}{2}\right)\right| \left|\reme\left(\frac{\lambda\beta_+}{2}\right)\right|\Bigg\}\\
 &\quad \leq \left|\eta_-\beta_-\right|\Bigg\{  \left[\frac{2^{\frac{d-1}2}}{\sqrt{\pi}\cdot{n!}} \left(\frac{\left|\lambda\beta_-\right|}{4}\right)^{2n+\frac{d}{2}-1}\right]\cdot \left[\frac{1}{1- \frac{\left|{\lambda\beta_-}\right|^2}{16(2n+d-2)}}\right]\left|\maine\left( \frac{\lambda\beta_+}{2} \right)\right| \\
 &\qquad + \left[\frac{2^{\frac{d-1}2}}{\sqrt{\pi}\cdot{n!}} \left(\frac{\left|\lambda\beta_-\right|}{4}\right)^{2n+\frac{d}{2}-1}\right]\cdot \left[\frac{1}{1- \frac{\left|{\lambda\beta_-}\right|^2}{16(2n+d-2)}}\right]\cdot \left[\frac{\left(\frac{|\lambda\beta_+|}{4}\right)^{2n+\frac{d}{2}}}{\left[\left(n+\frac{d}{2}\right)!\right]^2}\right] \cdot \left[ \frac{1}{1-\frac{|\lambda\beta_+|^2}{64(2n+d+2)^2}} \right]\\
 &\quad+  \left|\maino\left(\frac{\lambda\beta_-}{2}\right)\right|   \left[\frac{\left(\frac{|\lambda\beta_+|}{4}\right)^{2n+\frac{d}{2}}}{\left[\left(n+\frac{d}{2}\right)!\right]^2}\right] \cdot \left[ \frac{1}{1-\frac{|\lambda\beta_+|^2}{64(2n+d+2)^2}} \right]\Bigg\}\\
&\quad =: {\rm Err}_\Theta^n(\lambda).
\end{align*}
Both ${\rm Err}_\Theta^n(\lambda)$ and $\mathring{\Theta}^n(\lambda)$ involve only finitely many terms; hence, with each given $\lambda$, they can be computed explicitly by hand.

We are now at the stage of concluding:
\begin{equation}\label{denominator, arb d, conclusion}
\text{There exists $\lambda_\star \in ]2.5,3[$ such that $\Theta(\lambda_\star)=0$}.
\end{equation}
Indeed, by the triangle inequality and continuity of $\Theta$, it suffices to prove that 
\begin{equation}\label{Theta ><0}
\left[\mathring{\Theta} +{\rm Err}_\Theta^n\right](\lambda= 2.5) < 0 < \left[\mathring{\Theta} -{\rm Err}_\Theta^n\right](\lambda=3).
\end{equation}
With the help of $\textsc{Maple}^\text{\textregistered}$, we take $n=7$ and compute the two bounds in Eq.~\eqref{Theta ><0}. The results are tabulated below, which verify Eq.~\eqref{Theta ><0} and hence Eq.~\eqref{denominator, arb d, conclusion}.
\begin{table}[ht]
\caption{The denominator $\Theta$ has a root in $]2.5,3[$} 
\centering 
\begin{tabular}{c c c} 
\hline\hline 
Dimension & Value for $\left[\mathring{\Theta} +{\rm Err}_\Theta^n\right](2.5)$ & Value for $\left[\mathring{\Theta} -{\rm Err}_\Theta^n\right](3)$ \\ [0.5ex] 
\hline 
3 & -2.072008 & 6.951356 \\ 
4 & -1.682841 & 8.366543 \\
5 & -1.315936 & 6.921044  \\
6 & -1.107269 & 4.734511  \\
7 & -0.693811 & 2.530460 \\
8 & -0.408603 & 1.115177 \\ [1ex] 
\hline 
\end{tabular}
\label{table:arb d}
\end{table}

Next let us investigate $h^{(d+1)}(0)$, defined in the limiting sense:
\begin{eqnarray*}
h^{(d+1)}(0) = \lim_{r \rightarrow 0 }r^{1-\frac{d}{2}  }\frac{{\rm Im}\left\{\eta_{-}\beta_{-}\left[J_{\frac{d}{2}-1}\left(\frac{\lambda\beta_+r}{2}\right)\right]^\dagger J_{\frac{d}{2}}\left(\frac{\lambda\beta_+ }{2}\right) \right\}}{\Theta(\lambda)} =: \frac{\mathcal{N}(\lambda)}{\Theta(\lambda)}.
\end{eqnarray*}
As 
\begin{eqnarray*}
\lim_{r \rightarrow 0}\frac{\left[J_{\frac{d}{2}-1}\left(\frac{\lambda\beta_+ r}{2}\right)\right]^\dagger }{ r^{\frac{d}{2} - 1}} = \frac{1}{\Gamma(\frac{d}{2})} \left(\frac{\lambda \beta_{-}}{4}\right)^{\frac{d}{2}-1},
\end{eqnarray*}
we have
\begin{align*}
\mathcal{N}(\lambda) = \frac{1}{\Gamma(\frac{d}{2})} {\rm Im} \left( \eta_{-}\beta_{-}\left(\frac{\lambda \beta_{-}}{4}\right)^{\frac{d}{2}-1} J_{\frac{d}{2}}\left(\frac{\lambda\beta_+ }{2}\right)\right).
\end{align*}

To show that the numerator $\mathcal{N}(\lambda)$ is non-vanishing for $\lambda \in [2.5, 3]$, we split $\mathcal{N}(\lambda)$ into the approximating polynomial $\mathring{\mathcal{N}}^{n}(\lambda)$ and the remainder $\text{Err}^{n}_{\mathcal{N}}$,
where 
\begin{eqnarray*}
\mathring{\mathcal{N}}^{n}(\lambda) = {\rm Im}\left( \eta_{-}\beta_{-}\left(\frac{\lambda \beta_{-}}{4}\right)^{\frac{d}{2}-1} \mathring{J}^n_{\frac{d}{2}}\left(\frac{\lambda \beta_{+}}{2}\right)\right).
\end{eqnarray*}
Here again we use the polynomial $\mathring{J}^{n}_{\frac{d}{2}}\left(\frac{\lambda \beta_{+}}{2}\right)$ to approximate ${J}^{n}_{\frac{d}{2}}\left(\frac{\lambda \beta_{+}}{2}\right)$ in the numerator, with $n = 5$. Note that $\mathring{\mathcal{N}}^{n}(\lambda)$ is an explicit polynomial in $\lambda$ of degree $2n+d$, which is monotone decreasing on $[2.5, 3]$. One can bound the error   uniformly on $[2.5, 3]$:
\begin{align*}
\text{Err}_{\mathcal{N}}^{n}&:=\sup_{\lambda \in [2.5, 3]} \left|\mathcal{N}(\lambda) - \mathring{\mathcal{N}}^{n}(\lambda)\right| \\
&\leq \left\{\left|\eta_{-} \beta_{-} \left(\frac{3 \beta_{-}}{4}\right)^{\frac{d}{2}-1}\right|\right\} \cdot \left\{ \sup_{\lambda \in [2.5, 3]} \left|J_{\frac{d}{2}}\left(\frac{\lambda \beta_{+}}{2}\right) - \mathring{J}^{n}_{\frac{d}{2}}\left(\frac{\lambda \beta_{+}}{2}\right)\right|\right\}\\
&\leq \left\{\left| \eta_{-} \beta_{-} \left(\frac{3 \beta_{-}}{4}\right)^{\frac{d}{2}-1}\right| \right\}\left\{\left|E\left(\frac{3\beta_{+}}{2}, n, \frac{d}{2}\right)\right|\right\},
\end{align*}
where $E(z,n,\nu)$ is given below as in Appendix, Lemma~\ref{lem: bessel}:
\begin{eqnarray*}
E(z, n, \nu) := \begin{cases}
\left\{ \frac{\left(\frac{|z|}{2}\right)^{2n+\nu}}{n!(n+\nu)!}\right\} \cdot \left\{ \frac{1}{1-\frac{|z|^2}{4(n+1)^2}} \right\} &\text{ if $\nu$ is an integer},\\
\left\{\frac{2^n \sqrt{\pi}}{{n! (2n+2\nu+1)!!}} \left(\frac{|z|}{2}\right)^{2n+\nu}\right\}\cdot \left\{\frac{1}{1- \frac{|z|^2}{4(n+1)^2}}\right\}& \text{ if $\nu$ is a half-integer}.
\end{cases}
\end{eqnarray*}

\begin{table}[ht]
\caption{The numerator $\mathcal{N}(\lambda)$ is non-vanishing for $\lambda \in ]2.5,3[$} 
\centering 
\begin{tabular}{c c c} 
\hline\hline 
Dimension & Value for $\mathring{\mathcal{N}}^n(2.5)$ & Error bound $\text{Err}^n_{\mathcal{N}} $ \\ [0.5ex] 
\hline 
3 & -48.656672  & 1.266852 \\ 
4 & -55.063129 & 1.265336\\
5 & -51.368007 & 5.982517  \\
6 & -40.528560 & 3.647551  \\
7 & -26.851665 & 12.270795 \\
8 & -13.908808 & 5.810051 \\ [1ex] 
\hline 
\end{tabular}
\label{table:arb d_ball_numerator}
\end{table}

Table~\ref{table:arb d_ball_numerator}  summarises the values for $\mathring{\mathcal{N}}(2.5)$ and $\text{Err}_\mathcal{N}^n$ for $d \in \{3, \cdots, 8\}$  with accuracy up to 6 decimal points. This together with Lemma \ref{Lemma: bound_estimator} implies that
\begin{eqnarray*}
\sup_{\lambda \in [2.5, 3]} \mathcal{N}(\lambda) \leq \mathring{\mathcal{N}}^{n}(2.5) + \text{Err}^{n}_{\mathcal{N}} < 0,
\end{eqnarray*}
hence proves the thesis.   \end{proof}

\subsection{General $d$-dimensional $C^{2, \alpha}$ bounded domains}

Let $\Omega\Subset\mathbb{R}^{d}$ be a $C^{2, \alpha}$-domain and $x \in \Omega$ be an arbitrary point. Denote by $\mathcal{H}_{\lambda, \Omega}(z)$ be the hyperbolic development of the expected signature of a  $d$-dimensional Brownian motion starting from $z$ and stopped at the first exit time from $\Omega$ (scaled by $\lambda$). We are concerned with the domain-averaging hyperbolic development. Without loss of generality we may assume $x=0$, as $\Phi_\Omega(z) = \Phi_{\Omega-x}(z-x)$.

Recall from Definition~\ref{def: domain avg} the domain-averaging hyperbolic development (with integration understood entry-wise):
\begin{align*}
\overline{\mathcal{H}_{\lambda,\e,\Omega}}(z):= \int_{SO(d)} \mathcal{H}_{\lambda,\Omega_R}(z)\,\dd\chi(R).
\end{align*}
Here $
\overline{\mathcal{H}_{\lambda,\e,\Omega}}(z) \in \mathfrak{gl}(d+1;\R)$ for each $z \in \overline{\mathbf{B}(0,\e)}$, with $\e>0$ chosen so small that $\mathbf{B}(0,\e) \subset \bigcap_{S \in SO(d)} \Omega_S$. Compare with Proposition~\ref{propn: ODE for hyperbolic dev in d-dim} where $\Omega = \mathbf{D}$.

By Theorem~\ref{thm: pde for average, new}, $\overline{\mathcal{H}_{\lambda,\e,\Omega}}$ satisfies the same PDE as in Lemma~\ref{lem: PDE_H_d_ball} on $\mathbf{B}(0,\e)$:
\begin{theorem}[PDE for $\overline{\mathcal{H}_{\lambda,\e,\Omega}}$]\label{Thm:PDE_general_domain_hyperbolic_development} Let $\e >0$ and $\mathbf{B}(0,\e) \subset \bigcap_{R \in SO(d)} \Omega_R$. There exists $\lambda^* >0$ such that for every $\lambda \in [0, \lambda^*]$, 
\begin{equation*}
\Delta \overline{\mathcal{H}_{\lambda,\e, \Omega}}(z)=-2\lambda\sum_{i=1}^{d}He_{i}\frac{\partial \overline{\mathcal{H}_{\lambda,\e,\Omega}}}{\partial z^{i}}(z)-\lambda^{2}\left(\sum_{i=1}^{d}\left(He_{i}\right)^{2}\right)\overline{\mathcal{H}_{\lambda,\e,\Omega}}(z)\qquad \text{ for each } z \in \mathbf{B}(0,\e).
\end{equation*}
\end{theorem}

The domain-averaging hyperbolic development $\overline{\mathcal{H}_{\lambda,\e,\Omega}}(z)$ on a general domain $\Omega$ does not have the sparseness property as in Lemma~\ref{lem: sparseness}, in contrast to the case of the unit disc. To overcome this difficulty, we shall symmetrize $\overline{\mathcal{H}_{\lambda,\e,\Omega}}(z)$  with a reflected version of it. More precisely, consider
\begin{align*}
\Omega_\ast := \{(x_1,- x_{2}, \cdots, -x_{d}):\, (x_{1}, x_{2}, \cdots, x_{d}) \in \Omega\}.
\end{align*}
Assume as before $0 \in \Omega$; thus, there is an $\epsilon >0$ such that $\mathbf{D}_{\epsilon}$, the disc of radius $\e$, lies in $\Omega$. Set
\begin{eqnarray*}
\widetilde{\mathcal{H}_{\lambda,\e,\Omega}}(z):= \frac{\overline{\mathcal{H}_{\lambda,\e,\Omega}}(z) + \overline{\mathcal{H}_{\lambda,\e,\Omega_{\ast}}}(z)}{2}\qquad \text{ for } z \in \mathbf{D}_{\epsilon}.
\end{eqnarray*}

\begin{lemma}[Sparseness of $\widetilde{\mathcal{H}_{\lambda,\e,\Omega}}$]\label{lem:domain_averaging_development_reflected_domain} $\widetilde{\mathcal{H}_{\lambda,\e,\Omega}}$ evaluated at ${\bf r} = (r, 0, \cdots, 0)$ for $r \in ]0, \epsilon[$ has all components being zero, except possibly for the first and the last ones. 
\end{lemma}

\begin{proof}
The proof is similar to that of Lemma~\ref{lem: sparseness} for the unit disc case. Let $(B_t)_{t}$ denote the standard $d$-dimensional Brownian path.  
Consider the reflected path
\begin{equation}\label{reflected BM}
\hat{B}_t: = \left(B^{(1)}_{t}, - B^{(2)}_t, \cdots, -B^{(d)}_t\right)^\top,
\end{equation}
which is still a Brownian motion. If $B_t$ starts at $\mathbf{r} = (r, 0, \cdots, 0)^\top$, we may further require that $\hat{B}_{0} = \mathbf{r}$.
By the definition of $\widetilde{\mathcal{H}_{\lambda,\e,\Omega}}$, for $k \in \{2, \cdots, d\}$ we have
\begin{align*}
&\widetilde{\mathcal{H}_{\lambda,\e,\Omega}}^{(k)}(r,0, \cdots, 0)\\
&\quad :=\overline{\mathcal{H}_{\lambda, \e, \Omega}}^{(k)}(r) + \overline{\mathcal{H}_{\lambda, \e, \Omega_\ast}}^{(k)}(r)\\
&\quad = \int_{{\bf S}^{d-1}}\left\{\mathbb{E}^{r}\left[ \sum_{n \geq 0} \sum_{\omega \in \mathcal{E}_{2n}^{k}} \proj_{\omega}\left(S\left(B_{\tau_{\Omega_{\alpha}}}\right)\right)\right] + \mathbb{E}^{r}\left[ \sum_{n \geq 0} \sum_{\omega \in \mathcal{E}_{2n}^{k}} \proj_{\omega}\left(S\left(\hat{B}_{\tau_{\Omega_{\ast,\alpha}}}\right)\right)\right]\right\} \,\dd\mu(\alpha) \\
&\quad = 0.
\end{align*}
The last line holds as $\proj_{\omega}\left(S\left(B_{\tau_{\Omega_{\alpha}}}\right)\right) + \proj_{\omega}\left(S\left(\hat{B}_{\tau_{\Omega_{\ast,\alpha}}}\right)\right) = 0$ almost surely for $\omega \in \mathcal{E}_{2n}^{k}$.  \end{proof}

As before, write $\eta_\pm:= d+3 \pm \sqrt{d^2-10d+9}$, $\beta_\pm:=\sqrt{2d-6 \pm 2 \sqrt{d^2-10d+9}}$, and $r = |z|$. We have the following analogue of Theorem~\ref{thm: d-disk, PDE sol}.

\begin{theorem}\label{thm:H_reflected_domain_formula}
Let $\mathbf{D}_{\epsilon} \subset \Omega$ and $d \in \mathbb{N}\setminus\{1, 9\}$.  Denote the boundary condition for $\widetilde{\mathcal{H}_{\lambda,\e,\Omega}}$ at $r=\e$ as $\widetilde{\mathcal{H}_{\lambda,\e,\Omega}}(\epsilon) = [\mathfrak{p}, 0, \cdots, 0, \mathfrak{q}]^{\top}$. Then there exists $\lambda_\star >0$ such that for every $\lambda \in ]0, \lambda_\star[$,
\begin{align}\label{new, B}
 \widetilde{\mathcal{H}_{\lambda,\e,\Omega}}(z)  = \left[
\frac{z^{1}}{r}h_{ \lambda}^{(1)}(r),\,\frac{z^{2}}{r}h_{ \lambda}^{(1)}(r),\, \cdots, \, \frac{z^{d}}{r}h_{ \lambda}^{(1)}(r),\,h^{(d+1)}_{\lambda}(r) \right]^\top.
\end{align}
The functions $h^{(1)}_\lambda$, $h^{(d+1)}_\lambda$ and the Wr\"{o}nskian determinant $\mathcal{W}_\e$ are given by
\begin{align*}
&h^{(1)}_\lambda(r) = \frac{\left(\frac{r}{\e}\right)^{1-\frac{d}{2}}}{\mathcal{W}_\e}\cdot\Bigg\{\left[\eta_-\beta_-J_{\frac{d}{2}-1}\left(\frac{\lambda\beta_-\e}{2}\right)J_{\frac{d}{2}}\left(\frac{\lambda\beta_+r}{2}\right) -\eta_+\beta_+ J_{\frac{d}{2}}\left(\frac{\lambda\beta_-r}{2}\right)J_{\frac{d}{2}-1}\left(\frac{\lambda\beta_+\e}{2}\right)\right]\mathfrak{p} \\
&\quad  +\left[ - J_{\frac{d}{2}}\left(\frac{\lambda\beta_-\e}{2}\right)J_{\frac{d}{2}}\left(\frac{\lambda\beta_+ r}{2}\right)+  J_{\frac{d}{2}}\left(\frac{\lambda\beta_-r}{2}\right)J_{\frac{d}{2}}\left(\frac{\lambda\beta_+\e}{2}\right)\right]8d\mathfrak{q}\Bigg\};\\
&h^{(d+1)}_\lambda(r) =\frac{\left(\frac{r}{\e}\right)^{1-\frac{d}{2}}}{\mathcal{W}_\e}\cdot\Bigg\{\left[J_{\frac{d}{2}-1}\left(\frac{\lambda\beta_-\e}{2}\right)J_{\frac{d}{2}-1}\left(\frac{\lambda\beta_+r}{2}\right) - J_{\frac{d}{2}}\left(\frac{\lambda\beta_-r}{2}\right)J_{\frac{d}{2}-1}\left(\frac{\lambda\beta_+\e}{2}\right)\right]\frac{\mathfrak{p} \beta_{+}\beta_{-} \eta_{+}\eta_{-}}{8d} \\
&\quad + \left[  \eta_{-}\beta_{-} J_{\frac{d}{2}}\left(\frac{\lambda\beta_-\e}{2}\right)J_{\frac{d}{2}-1}\left(\frac{\lambda\beta_+ r}{2}\right)- \eta_{+}\beta_{+}J_{\frac{d}{2}}\left(\frac{\lambda\beta_+\e}{2}\right) J_{\frac{d}{2}-1}\left(\frac{\lambda\beta_-r}{2}\right) \right]\mathfrak{q}\Bigg\};\\
&\mathcal{W}_\e = \det \begin{bmatrix}
J_{\frac{d}{2}}\left( \frac{\lambda  \beta_+ \e}{2}\right) & J_{\frac{d}{2}}\left( \frac{\lambda  \beta_- \e}{2}\right)\\
\eta_+ \beta_+  J_{\frac{d}{2}-1} \left( \frac{\lambda \beta_+ \e}{2}\right) & \eta_- \beta_- J_{\frac{d}{2}-1} \left( \frac{\lambda \beta_- \e}{2}\right)
\end{bmatrix}.
\end{align*}
\end{theorem}

\begin{remark}
For $2\leq d \leq 8$ we can further simplify $h^{(1)}_\lambda$ and $h^{(d+1)}_\lambda$:
\begin{eqnarray*}
&h^{(1)}_\lambda(r) = \left(\frac{r}{\e}\right)^{1-\frac{d}{2}} \left\{ \frac{\mathfrak{p}\cdot  {\rm Im}\,\left[\eta_-\beta_-J_{\frac{d}{2}-1}\left(\frac{\lambda\beta_-\e}{2}\right)J_{\frac{d}{2}}\left(\frac{\lambda\beta_+r}{2}\right)\right]  +8d\mathfrak{q}\cdot {\rm Im}\,\left[J_{\frac{d}{2}}\left(\frac{\lambda\beta_-\e}{2}\right)J_{\frac{d}{2}}\left(\frac{\lambda\beta_+ r}{2}\right)\right] }{{\rm Im}\left[\eta_-\beta_-J_{\frac{d}{2}-1}\left(\frac{\lambda\beta_-\e}{2}\right) J_{\frac{d}{2}}\left(\frac{\lambda\beta_+\e}{2}\right) \right]}\right\};\\
&h^{(d+1)}_\lambda(r) =\left(\frac{r}{\e}\right)^{1-\frac{d}{2}} \left\{ \frac{\frac{1}{8d}\mathfrak{p}\cdot  {\rm Im}\,\left[\eta_{+}\beta_{+}\eta_-\beta_-J_{\frac{d}{2}-1}\left(\frac{\lambda\beta_-\e}{2}\right)J_{\frac{d}{2}-1}\left(\frac{\lambda\beta_+r}{2}\right)\right]  +\mathfrak{q}\cdot {\rm Im}\,\left[ \eta_{-}\beta_{-}J_{\frac{d}{2}}\left(\frac{\lambda\beta_+\e}{2}\right)J_{\frac{d}{2}-1}\left(\frac{\lambda\beta_- r}{2}\right)\right] }{{\rm Im}\left[\eta_-\beta_-J_{\frac{d}{2}-1}\left(\frac{\lambda\beta_-\e}{2}\right) J_{\frac{d}{2}}\left(\frac{\lambda\beta_+\e}{2}\right) \right]}\right\}.
\end{eqnarray*}
\end{remark}

\begin{proof}
Since both $\overline{\mathcal{H}_{\lambda, \e, \Omega}}$ and $\overline{\mathcal{H}_{\lambda, \e, \Omega_\ast}}$ satisfy the same PDE as in Theorem~\ref{Thm:PDE_general_domain_hyperbolic_development}, $\widetilde{\mathcal{H}_{\lambda,\e,\Omega}}$ satisfies also the same PDE on  $\mathbf{B}(0,\e) \subset \bigcap_{S \in SO(d)} \Omega_S$:
\begin{equation*}
\Delta \widetilde{\mathcal{H}_{\lambda,\e, \Omega}}(z)=-2\lambda\sum_{i=1}^{d}He_{i}\frac{\partial \widetilde{\mathcal{H}_{\lambda,\e,\Omega}}(z)}{\partial z^{i}}-\lambda^{2}\left(\sum_{i=1}^{d}\left(He_{i}\right)^{2}\right)\widetilde{\mathcal{H}_{\lambda,\e,\Omega}}(z)\qquad \text{ for each } z \in \mathbf{B}(0,\e).
\end{equation*}

The sparseness Lemma~\ref{lem:domain_averaging_development_reflected_domain} implies that components of $\overline{\mathcal{H}_{\lambda,\e,\Omega}}$ are all zero expect for the first and the last ones, just as in the case of unit disc. Let $\widetilde{\mathcal{H}_{\lambda,\e,\Omega}}(r) = \left(h^{(1)}_\lambda(r), 0, \cdots 0, h^{(d+1)}_\lambda(r)\right)^\top$. It then satisfies the PDE system \eqref{ODE one, d-dim},  \eqref{ODE two, d-dim} in Proposition~\ref{propn: ODE for hyperbolic dev in d-dim}, which differs from the unit disc case only in terms of boundary data.

Adapting the computations in the proof for Theorem~\ref{thm: d-disk, PDE sol}, we get
\begin{eqnarray*}
&&\widetilde{\mathcal{H}^{(1)}_{\lambda,\e, \Omega}}(r) = r^{1-\frac{d}{2}  }\left(C_{1} J_{\frac{d}{2}}\left( \frac{\lambda  \beta_+ r}{2}\right) +  C_{2}J_{\frac{d}{2}}\left( \frac{\lambda \beta_- r}{2}\right)\right),\\
&&\widetilde{\mathcal{H}^{(d+1)}_{\lambda,\e, \Omega}}(r) = - \frac{1}{8d}r^{1-\frac{d}{2}  }\left\{C_1 \eta_+ \beta_+  J_{\frac{d}{2}-1} \left( \frac{\lambda \beta_+ r}{2}\right) + C_2 \eta_-\beta_-  J_{\frac{d}{2}-1}\left( \frac{\lambda \beta_- r}{2}\right)\right\}. 
\end{eqnarray*}
for $C_1$ and $C_2$ to be specified. One may determine these constants from the boundary condition at $r=\e$. Indeed, we infer from Eq.~\eqref{bc, d dim general domain--new} that
\begin{equation}
\e^{1-\frac{d}{2}}
\begin{bmatrix}
J_{\frac{d}{2}}\cdot\left( \frac{\lambda  \beta_+ \e}{2}\right) & J_{\frac{d}{2}}\left( \frac{\lambda  \beta_- \e}{2}\right)\\
-\frac{1}{8d}\eta_+ \beta_{+}   J_{\frac{d}{2}-1} \left( \frac{\lambda \beta_+ \e}{2}\right) & -\frac{1}{8d}\eta_- \beta_- J_{\frac{d}{2}-1} \left( \frac{\lambda \beta_- \e}{2}\right)
\end{bmatrix}\cdot\begin{bmatrix}
C_1\\
C_2
\end{bmatrix} = \begin{bmatrix}
\mathfrak{p}\\
\mathfrak{q}
\end{bmatrix}.
\end{equation}

This algebraic system can be explicitly solved:
\begin{eqnarray*}
&&C_1 = \frac{\e^{\frac{d}{2}-1}}{\mathcal{W}_\e} \left\{\eta_-\beta_-J_{\frac{d}{2}-1}\left( \frac{\lambda\beta_-\e}{2} \right)\mathfrak{p} - 8d J_{\frac{d}{2}}\left( \frac{\lambda\beta_-\e}{2} \right)\mathfrak{q}\right\},\\
&&C_2 = \frac{\e^{\frac{d}{2}-1}}{\mathcal{W}_\e} \left\{-\eta_+\beta_+ J_{\frac{d}{2}-1}\left( \frac{\lambda\beta_+\e}{2} \right)\mathfrak{p} + 8d J_{\frac{d}{2}}\left( \frac{\lambda\beta_+\e}{2} \right)\mathfrak{q}\right\},
\end{eqnarray*}
where (comparing with Eq.~\eqref{wronskian, purely imaginary})
\begin{align*}
\mathcal{W}_\e &:= \det \begin{bmatrix}
J_{\frac{d}{2}}\left( \frac{\lambda  \beta_+ \e}{2}\right) & J_{\frac{d}{2}}\left( \frac{\lambda  \beta_- \e}{2}\right)\\
\eta_+ \beta_+  J_{\frac{d}{2}-1} \left( \frac{\lambda \beta_+ \e}{2}\right) & \eta_- \beta_- J_{\frac{d}{2}-1} \left( \frac{\lambda \beta_- \e}{2}\right)
\end{bmatrix}\\
& = 2i\,{\rm Im}\left\{\eta_-\beta_-\left[J_{\frac{d}{2}-1}\left(\frac{\lambda\beta_+\e}{2}\right)\right]^\dagger J_{\frac{d}{2}}\left(\frac{\lambda\beta_+\e}{2}\right) \right\}.
\end{align*}
This completes the proof.  \end{proof}

Finally, we arrive at the stage of proving Theorem~\ref{thm: main} (reproduced below).

\begin{theorem*}
Let $\Omega$ be a bounded  $C^{2,\alpha}$-domain in $\R^d$, for some $d \in \{2, \ldots, 8\}$ and $\alpha > 0$. The expected signature $\Phi$ of a Brownian motion stopped upon the first exit time from $\Omega$ has finite radius of convergence everywhere on $\Omega$. 
\end{theorem*}
\begin{proof}
Our arguments will essentially be an adaptation of the proof of Theorem~\ref{thm: d-dim, finite ROC}. We shall only indicate necessary modifications.

First of all, Step~1 in the proof of Theorem~\ref{thm: main} ($d=2$ case) in \S\ref{sec: proof of main thm} remains valid for $\Omega \Subset \R^d$. This is the only place that we need the $C^{2,\alpha}$-regularity of the domain.  We conclude that  \begin{align*}
\left\|\proj_n \Phi\right\|_{C^2(\Omega)} \leq C^n
\end{align*}
for $C$ depending only on the geometry of $\Omega$, the H\"{o}lder index $\alpha$, and the dimension $d$. 

Consider now $\widetilde{\mathcal{H}_{\lambda, \e, \Omega}} = \left(h^{(1)}_{\lambda}, \cdots, h^{(d+1)}_{\lambda}\right)^{\top}$. We show that there  exists $\lambda_{\star} > 0$ for which $\lim_{\lambda \rightarrow \lambda_{\star} }h^{(d+1)}_{\lambda}(0) = \infty$. Label as before the boundary conditions at $r=\e$ by
\begin{equation}\label{bc, d dim general domain--new}
h^{(1)}_\lambda(\e) = \mathfrak{p} \quad \text{ and } \quad h^{(d+1)}_\lambda(\e)=\mathfrak{q},
\end{equation}
where, as a precaution, $\mathfrak{p}$ and $\mathfrak{q}$ depend in general on $\lambda$ and $\e$. 

By Theorem~\ref{thm:H_reflected_domain_formula} and the ensuing remark,   we have the closed-form expression
\begin{eqnarray*}
h^{(d+1)}_\lambda(r) &=\left(\frac{r}{\e}\right)^{1-\frac{d}{2}} \left\{ \frac{\frac{1}{8d}\mathfrak{p}\cdot  {\rm Im}\,\left[\eta_{+}\beta_{+}\eta_-\beta_-J_{\frac{d}{2}-1}\left(\frac{\lambda\beta_-\e}{2}\right)J_{\frac{d}{2}-1}\left(\frac{\lambda\beta_+r}{2}\right)\right]  +\mathfrak{q}\cdot {\rm Im}\,\left[ \eta_{-}\beta_{-}J_{\frac{d}{2}}\left(\frac{\lambda\beta_+\e}{2}\right)J_{\frac{d}{2}-1}\left(\frac{\lambda\beta_- r}{2}\right)\right] }{{\rm Im}\left[\eta_-\beta_-J_{\frac{d}{2}-1}\left(\frac{\lambda\beta_-\e}{2}\right) J_{\frac{d}{2}}\left(\frac{\lambda\beta_+\e}{2}\right) \right]}\right\}
\end{eqnarray*}
for $d \in \{2, \dots, 8\}$. Set $h^{(d+1)}_\lambda(0):= \lim_{r\searrow 0} h^{(d+1)}_\lambda(r)$; by properties of the Bessel function,
\begin{align*}
h^{(d+1)}_\lambda(0) =\frac{\mathcal{N}_{\e}(\lambda)}{\Theta_{\e}(\lambda)}
\end{align*}
with the numerator
\begin{align}\label{NNN}
\mathcal{N}_{\e}(\lambda) &= \underset{r \rightarrow 0}{\lim}\left(\frac{r}{\e}\right)^{1-\frac{d}{2}}  \Bigg\{\frac{\mathfrak{p}}{8d}\cdot  {\rm Im}\,\left[\eta_{+}\beta_{+}\eta_-\beta_-J_{\frac{d}{2}-1}\left(\frac{\lambda\beta_-\e}{2}\right)J_{\frac{d}{2}-1}\left(\frac{\lambda\beta_+r}{2}\right)\right] \nonumber\\
&\qquad\qquad\qquad\qquad+\mathfrak{q}\cdot {\rm Im}\,\left[ \eta_{-}\beta_{-}J_{\frac{d}{2}}\left(\frac{\lambda\beta_+\e}{2}\right)J_{\frac{d}{2}-1}\left(\frac{\lambda\beta_- r}{2}\right)\right]\Bigg\}\nonumber\\
&=\mathfrak{q} \cdot \underbrace{\frac{1}{\Gamma(\frac{d}{2})}\cdot{\rm Im}\,\left[\eta_{-}\beta_{-}J_{\frac{d}{2}}\left(\frac{\lambda\beta_+\e}{2}\right)\left(\frac{\lambda\beta_- \e}{4}\right)^{\frac{d}{2}-1}\right]}_{\mathcal{N}^{(1)}_{\e}(\lambda)}  \nonumber\\
&\qquad\qquad\qquad\qquad+\mathfrak{p}\cdot  \underbrace{\frac{1}{\Gamma(\frac{d}{2})}\cdot{\rm Im}\,\left[\frac{1}{8d} \eta_{+}\beta_{+}\eta_-\beta_-J_{\frac{d}{2}-1}\left(\frac{\lambda\beta_-\e}{2}\right)\left(\frac{\lambda\beta_+\e}{4}\right)^{\frac{d}{2}-1} \right]}_{\mathcal{N}^{(2)}_{\e}(\lambda)}
\end{align}
and the denominator $$\Theta_{\e}(\lambda) = {\rm Im}\left[\eta_-\beta_-J_{\frac{d}{2}-1}\left(\frac{\lambda\beta_-\e}{2}\right) J_{\frac{d}{2}}\left(\frac{\lambda\beta_+\e}{2}\right) \right].$$

From Lemma~\ref{lem: hyperbolic dev}  and Corollary~\ref{cor: new} we infer that
\begin{equation*}
\left[h^{(1)}_\lambda(r), 0, \cdots, 0 , h^{(d+1)}_\lambda(r)\right]^\top \in \left\{ x \in \mathbb{R}^{d+1} :\,  x^{d+1}\geq \left|x^1\right|\text{ and } x^{d+1} \geq 1 \right\}\qquad\text{for all } r \in ]0,\e].
\end{equation*} In particular, when $\mathfrak{p}\neq 0$ we have $\frac{\mathfrak{q}}{|\mathfrak{p}|} \geq 1$ and $ \mathfrak{q} \geq 1$ in the almost sure sense, thanks to Eq.~\eqref{bc, d dim general domain--new}.

The proof of Theorem~\ref{thm: d-dim, finite ROC} shows that there is a root $\lambda_\star$ for the denominator $\Theta_{\e}$, with $\lambda_\star \in \left]\frac{2.5}{\e}, \frac{3}{\e} \right[$. It thus suffices to check that the numerator $\mathcal{N}_{\e}(\lambda)$ is non-vanishing.

\noindent
{\bf Case 1: $\mathfrak{p}=0$.} Then we have $$\mathcal{N}_{\e}(\lambda) = \mathfrak{q}\cdot {\rm Im}\left[ \eta_{-}\beta_{-}J_{\frac{d}{2}}\left(\frac{\lambda\beta_+\e}{2}\right)\left(\frac{\lambda\beta_- \e}{4}\right)^{\frac{d}{2}-1}\right],$$ which is essentially the same as the numerator for the unit ball case (modulo scaling). By the proof of Theorem~\ref{thm: d-dim, finite ROC}, for $\mu \in ]2.5, 3[$ one has $\eta_{-}\beta_{-}J_{\frac{d}{2}}\left(\frac{\mu\beta_+}{2}\right)\left(\frac{\mu\beta_- }{4}\right)^{\frac{d}{2}-1}\neq 0$. Using a change of variable and the geometrical constraint $\mathfrak{q}\geq 1$, one may easily show that  $\mathcal{N}_{\e}(\lambda)$ is bounded away from zero for all $\lambda \in \left]\frac{2.5}{\e}, \frac{3}{\e}\right[$ . 

\noindent
{\bf Case 2: $\mathfrak{p}\neq 0$.} Since the 2-dimensional case has been treated separately, we assume from now on $d \in \{3,4,\ldots,8\}$. Once we establish the following \emph{claim}:
\begin{equation}\label{new-claim to conclude}
\left| \mathcal{N}_{\e}^{(1)}(\lambda) \right| -\left|  \mathcal{N}_{\e}^{(2)}(\lambda)\right| > 0 \qquad\text{ for all } \lambda \in \left[\frac{2.5}{\e}, \frac{3}{\e}\right],
\end{equation}
we can immediately conclude the proof using the triangle inequality plus the geometrical constraint $\mathfrak{q}\geq |\mathfrak{p}|>0$. Indeed, it holds that
\begin{align*}
\left|\mathcal{N}_{\e}(\lambda)\right| &=  \left|\mathfrak{q} \mathcal{N}_{\e}^{(1)}(\lambda) + \mathfrak{p} \mathcal{N}_{\e}^{(2)}(\lambda)\right|\\
&=\mathfrak{p} \left|\frac{\mathfrak{q}}{\mathfrak{p}} \mathcal{N}_{\e}^{(1)}(\lambda) +  \mathcal{N}_{\e}^{(2)}(\lambda)\right| \\
&\geq  \mathfrak{p} \left(\left|\frac{\mathfrak{q}}{\mathfrak{p}} \mathcal{N}_{\e}^{(1)}(\lambda) \right| -\left|  \mathcal{N}_{\e}^{(2)}(\lambda)\right| \right)\\
&\geq  \mathfrak{p} \left(\left| \mathcal{N}_{\e}^{(1)}(\lambda) \right| -\left|  \mathcal{N}_{\e}^{(2)}(\lambda)\right| \right)\\ 
&> 0.
\end{align*}

The proof of \emph{claim}~\eqref{new-claim to conclude} will be presented in the appendix. See Lemma~\ref{Lem_d_domain_numerator}.  \end{proof}

\subsection{Concluding remarks}
We have proved in this paper the finiteness of radius of convergence of the expected signature for stopped Brownian motions in dimensions $2 \leq d \leq 8$. The restriction on the range of  $d$ mainly arises from the ODE system for the hyperbolic development $\mathcal{H}_{\lambda,\Omega}$ derived in Proposition~\ref{propn: ODE for hyperbolic dev in d-dim} (see also Theorem~\ref{thm: d-disk, PDE sol}). This ODE has explicit solution involving key parameters $\eta_\pm:= d+3 \pm \sqrt{d^2-10d+9}$ and $\beta_\pm:=\sqrt{2d-6 \pm 2 \sqrt{d^2-10d+9}}$, which are real for $d \geq 9$, whence the solutions demonstrate qualitatively different behaviours. We shall leave this point for future investigations.

\appendix
\section{ }
\numberwithin{equation}{section}

In this appendix, we first collect a (crude) upper bound for the Bessel functions:

\begin{lemma}\label{lem: bessel}
Consider Bessel functions of the first kind, $J_\nu(z)$, and their Taylor expansion:
\begin{align*}
{J}_\nu(z)  := \left(\frac{z}{2}\right)^\nu \sum_{k=0}^{\infty} (-1)^k \frac{\left(\frac{z^2}{4}\right)^k}{k! \G(k+\nu+1)}\qquad \text{ for $z \in \C$}.
\end{align*} 
Write $J_{\nu}(z) = \mathring{J}_\nu^n(z) + \mathscr{R}_\nu^n(z) := \sum_{k=0}^{n-1} + \sum_{k=n}^{\infty}$. The remainder term $\mathscr{R}_\nu^n(z)$ can be estimated by
\begin{align*}
\left|\mathscr{R}_\nu^n(z)\right| \leq E(z, n, \nu)\qquad \text{for $|z|<2(n+1)$}, 
\end{align*}
where
\begin{eqnarray*}
E(z, n, \nu) := \begin{cases}
\left\{ \frac{\left(\frac{|z|}{2}\right)^{2n+\nu}}{n!(n+\nu)!}\right\} \cdot \left\{ \frac{1}{1-\frac{|z|^2}{4(n+1)^2}} \right\} &\text{ if $\nu$ is an integer},\\
\left\{\frac{2^n}{{n! (2n+2\nu+1)!!}} \left(\frac{|z|}{2}\right)^{2n+\nu}\right\}\cdot \left\{\frac{1}{1- \frac{|z|^2}{4(n+1)^2}}\right\} & \text{ if $\nu$ is a half-integer}.
\end{cases}
\end{eqnarray*}
\end{lemma}

\begin{proof} 
From the Taylor expansion for $J_\nu$ we get
\begin{align*}
\left|\mathscr{R}_\nu^n(z)\right| \leq \left(\frac{|z|}{2}\right)^\nu \sum_{k=n}^\infty \frac{\left(\frac{|z|}{2}\right)^{2k}}{k!\G(k+\nu+1)}= \left(\frac{|z|}{2}\right)^{2n+\nu} \sum_{k=0}^\infty \frac{\left(\frac{|z|}{2}\right)^{2k}}{(n+k)!\G(n+k+\nu+1)}.
\end{align*}
Then, we deduce from the functional identity $\G(z+1)=z\G(z)$ that
\begin{align*}
\left|\mathscr{R}_\nu^n(z)\right| &\leq \left(\frac{|z|}{2}\right)^{2n+\nu}\frac{1}{n! \G(n+\nu+1)} \sum_{k=0}^\infty\frac{\left(\frac{|z|}{2}\right)^{2k}}{(n+1)^{2k}}\\
&= \left(\frac{|z|}{2}\right)^{2n+\nu}\frac{1}{n! \G(n+\nu+1)} \cdot \frac{1}{1-\frac{|z|^2}{4(n+1)^2}}\qquad\text{ for } |z|<2(n+1).
\end{align*}
The proof is complete in view of the special values for $\G(z)$ at integers and half-integers.  \end{proof}

The following simple lemma is rather helpful when estimating errors in Taylor expansions: 
\begin{lemma}\label{Lemma: bound_estimator}
Let $f: [a, b] \rightarrow \mathbb{R}$ be a non-increasing function. Let $g: [a, b] \rightarrow \mathbb{R}$ be  such that $\sup_{r \in [a, b]} |f(r) - g(r)| \leq C$ for some constant $C$. Then $\sup_{r \in [a, b]} g(r) \leq f(a) + C$. 
\end{lemma}
\begin{proof}
By the monotonicity of $f$, for any $r \in [a, b]$,
\begin{eqnarray*}
g(r) = g(r) - f(r) + f(r) \leq g(r)-f(r)+f(a) \leq |g(r) - f(r)| + f(a).
\end{eqnarray*}
Therefore, taking the supremum over $[a,b]$, we get
\begin{eqnarray*}
\sup_{r \in [a, b]} g(r) \leq \sup_{r \in [a, b]} |g(r) - f(r)| + f(a) \leq C + f(a).
\end{eqnarray*}
\end{proof}

In the computer-assisted proof below, the supplementary $\textsc{Maple}^\text{\textregistered}$ codes can be found in the Github repository \texttt{https://github.com/hello0630/ESigStoppedBM}.

\begin{lemma}\label{lem:d=9case}
Let $\mathbf{D} = \mathbf{D}^{d}$ with $d = 9$. Let $\lambda_\star$ denote the radius of convergence of $\overline{\mathcal{H}_{\lambda, \mathbf{D}}}$. Then for all $[0, \lambda_\star[$ we have
\begin{align}
    \overline{\mathcal{H}_{\lambda, \mathbf{D}}}(z)  = \left[
\frac{z^{1}}{r}h_{ \lambda}^{(1)}(r),\frac{z^{2}}{r}h_{ \lambda}^{(1)}(r),\, \cdots ,\,\frac{z^{d}}{r}h_{ \lambda}^{(1)}(r),\, h^{(d+1)}_{\lambda}(r) \right]^\top,
\end{align}
where
\begin{align*}
& h^{(1)}(r) = \frac {1}{{t}^{8}} \Bigg\{ -3\,\lambda\,t \left( {\lambda}^{2} \left( {\lambda}^{2}{\it C_4}+{\it C_2} \right) {t}^{4}+ \left( -35\,{\lambda}^{2}{\it C_4}-5\,{\it C_2} \right) {t}^{2} +105\,{\it C_4} \right) \cos \left( \lambda\, \sqrt{3}t \right)\\
&\quad +6\, \sqrt{3}\sin \left( \lambda\, \sqrt{3}t \right)  \left( {\lambda}^{2} \left( \frac{5}{2}\,{\lambda}^{2}{\it C_4}+{\it C_2} \right) {t}^{4}+ \left( -{\frac {70\,{\lambda}^{2}{\it C_4}}{3}}-\frac{5}{6}\,{\it C_2} \right) {t}^{2}+{\frac {35\,{\it C_4}}{2}} \right)  \Bigg\} ,\\
&h^{(d+1)}(r)= \frac {1}{\lambda\,{t}^{7}} \Bigg\{ - \left( {t}^{4}{\lambda}^{6}{\it C_4}+{t}^{2} \left( {\it C_2}\,{t}^{2}-13\,{\it C_4} \right) {\lambda}^{4}+ \left( 5\,{\it C_2}\,{t}^{2}+10\,{\it C_4} \right) {\lambda}^{2}
\mbox{}-5\,{\it C_2} \right)  \sqrt{3}\sin \left( \lambda\, \sqrt{3}t \right) \\
&\qquad-9\,t\cos \left( \lambda\, \sqrt{3}t \right) \lambda\, \left( {t}^{2}{\lambda}^{4}{\it C_4}-10/3\,{\lambda}^{2}{\it C_4}+5/3\,{\it C_2} \right)  \Bigg\},
\end{align*}
with the constants
\begin{align*}
&C_2 = \frac{3\lambda}{ \mathcal{W}} \left(  \left( {\lambda}^{5}-35\,{\lambda}^{3}+105\,\lambda \right) \cos \left( \lambda\, \sqrt{3} \right) -5\, \left( {\lambda}^{4}-{\frac {28\,{\lambda}^{2}}{3}}+7 \right) \sin \left( \lambda\, \sqrt{3} \right)  \sqrt{3}\right) ,\\
&C_4 =  \frac{-3 }{ \mathcal{W}}\lambda\, \left(  \left( {\lambda}^{3}-5\,\lambda \right) \cos \left( \lambda\, \sqrt{3} \right) -2\, \left( {\lambda}^{2}-5/6 \right) \sin \left( \lambda\, \sqrt{3} \right)  \sqrt{3}  \right)
\end{align*}
and the Wr\"{o}nskian
\begin{align*}
\mathcal{W}\equiv \mathcal{W}[\lambda] &= \left( -324\,{\lambda}^{6}+4410\,{\lambda}^{4}-8550\,{\lambda}^{2}+1575 \right)  \left( \cos \left( \lambda\, \sqrt{3} \right)  \right) ^{2}\\
&\qquad-18\,\lambda\, \sqrt{3}\sin \left( \lambda\, \sqrt{3} \right)  \left( {\lambda}^{6}-50\,{\lambda}^{4}+250\,{\lambda}^{2}-175 \right) \cos \left( \lambda\, \sqrt{3} \right) \\
&\qquad+27\,{\lambda}^{8}+54\,{\lambda}^{6}-2385\,{\lambda}^{4}+3825\,{\lambda}^{2}-1575.
\end{align*}
\end{lemma}
\begin{proof}
The proof is similar to that of Theorem \ref{thm: d-disk, PDE sol}. We use $\textsc{Maple}^\text{\textregistered}$ to obtain the fundamental solution and use the boundary condition to determine the constants. One can refer the calculation details to \texttt{CheckSolutionGeneralDomain\_d=9.mw} in the Github repository. 
\end{proof}

Next, let us prove Eq.~\eqref{N, mu, e, WTS} in Step~5 of the proof for Theorem~\ref{thm: main} ---
\begin{align*}
\mathcal{N}(\mu;\mathfrak{e}) := {\rm Im}\left\{ \mathfrak{e} \alpha J_1(\mu\zeta) - J_0(\mu\zeta) \right\} \neq 0\quad\text{ for all } \mu \in [2.5,3],\,
|\mathfrak{e}|\geq 1.
\end{align*}

\begin{lemma}\label{lem: appendix}
Let $\zeta = \frac{\sqrt{-2+2\sqrt{7}i}}{2}$ (which is a root for $p(z) = z^{4} + z^2 +2$) and $\alpha = \frac{1}{2}\zeta^{3} + \zeta$. Then
\begin{equation*}\mathcal{N}(\mu;\mathfrak{e}) := {\rm Im}\left\{ \mathfrak{e} \alpha J_1(\mu\zeta) - J_0(\mu\zeta) \right\} \neq 0\quad\text{ for all } \mu \in [2.5,3] \text{ and } |\mathfrak{e}|\geq 1.
\end{equation*}
\end{lemma}

\begin{proof}
It has been established along the proof of Lemma~8 in \cite{boedihardjo2021expected} that
\begin{align*}
{\rm Im} \left\{\bar{\alpha} J_{1}\left(\mu \bar{\zeta}\right)\right\} < - 1.3\quad\text{ for all } \mu \in [2.5,3].
\end{align*}
We shall prove the following inequalities:
\begin{align}\label{new--to prove, ineqs}
0<{\rm Im} \left\{J_{0}\left(\mu \bar{\zeta}\right)\right\} < -{\rm Im} \left\{\bar{\alpha} J_{1}\left(\mu \bar{\zeta}\right)\right\} \quad\text{ for all } \mu \in [2.5,3].
\end{align}
Indeed, as $|\mathfrak{e}|\geq 1$, we can conclude by estimating
\begin{align*}
\left|\mathcal{N}(\mu;\mathfrak{e})\right|&= \left|-\mathfrak{e} \cdot {\rm Im} \left\{\bar{\alpha} J_{1}\left(\mu \bar{\zeta}\right)\right\} + {\rm Im} \left\{J_{0}\left(\mu \bar{\zeta}\right)\right\}\right|\\
& \geq \bigg|\left|\mathfrak{e} \cdot {\rm Im} \left\{\bar{\alpha} J_{1}\left(\mu \bar{\zeta}\right)\right\}\right| - \left| {\rm Im} \left\{J_{0}\left(\mu \bar{\zeta}\right)\right\}\right|\bigg|\\
&= \left|\underbrace{\left(\left|\mathfrak{e}\right|-1\right)}_{\geq 0} \cdot\underbrace{{\rm Im} \left\{\bar{\alpha} J_{1}\left(\mu \bar{\zeta}\right)\right\}}_{<-1.3}+ \underbrace{{\rm Im} \left\{J_{0}\left(\mu \bar{\zeta}\right)\right\}+{\rm Im} \left\{\bar{\alpha} J_{1}\left(\mu \bar{\zeta}\right)\right\}}_{<0} \right|\\
&>0.
\end{align*}

To prove the second inequality in Eq.~\eqref{new--to prove, ineqs}, we write
\begin{align*}
&{\rm Im} \left\{\bar{\alpha} J_{1}\left(\mu \bar{\zeta}\right)\right\} +  {\rm Im} \left\{J_{0}\left(\mu \bar{\zeta}\right)\right\}\\
&\qquad = \underbrace{{\rm Im} \left\{\bar{\alpha} \mathring{J}^{n}_{1}\left(\mu \bar{\zeta}\right)\right\} +  {\rm Im} \left\{\mathring{J}^{n}_{0}\left(\mu \bar{\zeta}\right)\right\} }_{=:\mathring{\mathcal{N}}^{n}(\mu)}  + \underbrace{{\rm Im} \left\{\bar{\alpha} \mathcal{R}^{n}_{1}\left(\mu \bar{\zeta}\right)\right\} +  {\rm Im} \left\{\mathcal{R}^{n}_{0}\left(\mu \bar{\zeta}\right)\right\} }_{=:\,{\mathscr{R}}^{n}(\mu)}.
\end{align*}
The notation on the right-hand side agrees with the decomposition of $J_{\nu}^{n}$ into $\mathring{J}^n_{\nu}+\mathscr{R}_\nu^n$ via Taylor expansions. $\mathring{\mathcal{N}}^{n}$ is a polynomial in $\mu$ of degree $2n+1$. In the remaining parts of the proof we shall choose $n = 6$.

The bounds in Lemma~\ref{lem: bessel} yield that
\begin{eqnarray*}
\sup_{\mu \in [0, 3]} \max\left(\left|\mathscr{R}_{0}^{n}\left(\mu \bar{\zeta}\right)\right|, \left|\mathscr{R}_{1}^{n}\left(\mu \bar{\zeta}\right)\right|\right) \leq 0.0006367,
\end{eqnarray*}
which gives us
\begin{eqnarray*}
\sup_{\mu \in [0, 3]}\left|\mathscr{R}^{n}(\lambda)\right| \leq (|\alpha| + 1)\max\left(\left|\mathscr{R}_{0}^{n}\left(\mu \bar{\zeta}\right)\right|, \left|\mathscr{R}_{1}^{n}\left(\mu \bar{\zeta}\right)\right|\right) \leq 0.0011395.
\end{eqnarray*}

To estimate $\mathring{\mathcal{N}}^{n}$, let us write it as a polynomial of degree $n$ based at the point $\mu = 2.5$:
\begin{eqnarray*}
\mathring{\mathcal{N}}^{n}(\mu) = \sum_{i = 0}^{2n+1} C_i (\mu - 2.5)^{i}.
\end{eqnarray*}
The coefficients can be found, \emph{e.g.}, using $\textsc{Maple}^\text{\textregistered}$.  The first four $C_i$ are negative, with $C_{0}<-0.119150$, $C_1 < -0.169$, $C_2 < - 0.184$,  $C_3 < -0.049$, but $C_{4} > 0$. Let $\mathscr{T}^{3}(\mu) = \sum_{i=0}^{3}C_i(\mu - 2.5)^i$. For $\mu \in [2.5, 3]$, one may bound
\begin{eqnarray*}
\left|\mathring{\mathcal{N}}^{n}(\mu) - \mathscr{T}^{3}(\mu)\right|\leq \sum_{i = 4}^{12} |C_{i}| \cdot \left(\frac{1}{2}\right)^{i} \leq 0.00036.
\end{eqnarray*}
Note that $\mathscr{T}^{3}$ is a decreasing function. By Lemma~\ref{Lemma: bound_estimator}  we can further bound $\mathscr{T}^{n}(\mu)$ by
\begin{eqnarray*}
\sup_{\mu \in [2.5, 3] }\mathscr{T}^{n}(\mu) \leq \mathscr{T}^{n}(2.5) + \sup_{\mu \in [2.5, 3]} \left|\mathscr{T}^{n}(\mu) - \mathscr{T}^{3}(\mu)\right|.
\end{eqnarray*}

By triangle inequality and the explicit value $\mathscr{T}^{n}(2.5) = -0.1181564882$, we notice that
\begin{eqnarray*}
\sup_{\mu \in [2.5, 3]} \left\{\mathscr{T}^{n}(\mu) + {\mathscr{R}}^{n}(\mu)\right\}\leq -0.1181564882 + 0.00038 + 0.0011395 <0.
\end{eqnarray*}
The proof is now complete.  \end{proof}

Finally, let us verify the \emph{claim}~\eqref{new-claim to conclude} towards the end of the proof of Theorem~\ref{thm: main}.

\begin{lemma}\label{Lem_d_domain_numerator} For $d \in \{2, 3, \cdots, 8\}$, as in Eq.~\eqref{NNN} (with $\mu=\e\lambda$) we set
\begin{align*}
&\mathcal{N}^{(1)}(\mu) := \frac{1}{\Gamma(\frac{d}{2})}{\rm Im}\,\left[\eta_{-}\beta_{-}J_{\frac{d}{2}}\left(\frac{\mu\beta_+}{2}\right)\left(\frac{\mu\beta_- }{4}\right)^{\frac{d}{2}-1}\right],\\
&\mathcal{N}^{(2)}(\mu) := \frac{1}{\Gamma(\frac{d}{2})}{\rm Im}\,\left[\frac{1}{8d} \eta_{+}\beta_{+}\eta_-\beta_-J_{\frac{d}{2}-1}\left(\frac{\mu\beta_-}{2}\right)\left(\frac{\mu\beta_+}{4}\right)^{\frac{d}{2}-1} \right].
\end{align*}
Then it holds that
\begin{align*}
\left| \mathcal{N}^{(1)}(\mu) \right| -\left|  \mathcal{N}^{(2)}(\mu)\right| > 0\qquad \text{  for each } \mu \in [2.5, 3].
\end{align*}
\end{lemma}

\begin{proof}
The case $d = 2$ has already been treated in Lemma~\ref{lem: appendix}; we only consider $d \in \{3, 4, \cdots, 8\}$ from now on.  The notations in the proof are consistent with Lemma~\ref{lem: bessel}. Also, as shown in the proof of Theorem~\ref{thm: d-dim, finite ROC}, $\mathcal{N}^{(1)}(\mu) < 0$ for any $\mu \in [2.5, 3]$. It is  enough to prove that
\begin{enumerate}
\item
$\mathcal{N}^{(2)}(\mu) >0$;
\item
$\mathcal{N}(\mu):=\mathcal{N}^{(1)}(\mu) + \mathcal{N}^{(2)}(\mu) < 0$ for each $\mu \in [2.5, 3]$.
\end{enumerate}

For the first claim (1), we use $\frac{1}{\Gamma\left(\frac{d}{2}\right)}{\rm Im}\,\left[\frac{1}{8d} \eta_{+}\beta_{+}\eta_-\beta_-\mathring{J}^{n}_{\frac{d}{2}-1}\left(\frac{\mu\beta_-}{2}\right)\left(\frac{\mu\beta_+}{4}\right)^{\frac{d}{2}-1} \right]$ to approximate $\mathcal{N}^{(2)}$. The remainder is bounded by $\frac{1}{\Gamma(\frac{d}{2})}\left|\frac{1}{8d} \eta_{+}\beta_{+}\eta_-\beta_- \left(\frac{\mu\beta_+}{4}\right)^{\frac{d}{2}-1}\right|\cdot E\left(\frac{\mu\beta_-}{2}, n, \frac{d}{2}-1\right)$.

Similarly, we use Taylor approximation to approximate $\mathcal{N}(\mu)$. The truncated polynomial up to the  $n^{\text{th}}$ term is
\begin{align*}
\mathring{\mathcal{N}}^{n}(\mu) &:= \frac{1}{\Gamma(\frac{d}{2})}{\rm Im}\,\left[\eta_{-}\beta_{-}\mathring{J}^{n}_{\frac{d}{2}}\left(\frac{\mu\beta_+}{2}\right)\left(\frac{\mu\beta_- }{4}\right)^{\frac{d}{2}-1}\right] \\
&\qquad\qquad+ \frac{1}{\Gamma(\frac{d}{2})}{\rm Im}\,\left[\frac{1}{8d} \eta_{+}\beta_{+}\eta_-\beta_-\mathring{J}^{n}_{\frac{d}{2}-1}\left(\frac{\mu\beta_-}{2}\right)\left(\frac{\mu\beta_+}{4}\right)^{\frac{d}{2}-1} \right].
\end{align*}
The error can be  estimated uniformly as follows:
\begin{align*}
&\sup_{\mu \in [2.5,3]}\,\left|\mathring{\mathcal{N}}^{n}(\mu) - {\mathcal{N}}(\mu)\right| \\
&\quad\leq \text{Err}^n_{\mathcal{N}} := \sup_{\mu \in [2.5,3]}\,\Bigg\{\frac{1}{\Gamma\left(\frac{d}{2}\right)}\left|\eta_{-}\beta_{-} \left(\frac{\mu\beta_- }{4}\right)^{\frac{d}{2}-1}\right|E\left(\frac{\mu\beta_+}{2}, n, \frac{d}{2}\right)\\
&\qquad\qquad\qquad\qquad+\frac{1}{\Gamma\left(\frac{d}{2}\right)}\left|\frac{1}{8d} \eta_{+}\beta_{+}\eta_-\beta_- \left(\frac{\mu\beta_+}{4}\right)^{\frac{d}{2}-1}\right| E\left(\frac{\mu\beta_-}{2}, n, \frac{d}{2}-1\right)\Bigg\}\\
&\quad\qquad\qquad=\frac{1}{\Gamma\left(\frac{d}{2}\right)}\left|\eta_{-}\beta_{-} \left(\frac{3\beta_- }{4}\right)^{\frac{d}{2}-1}\right|E\left(\frac{3\beta_+}{2}, n, \frac{d}{2}\right)\\
&\quad\qquad\qquad\qquad\quad+\frac{1}{\Gamma\left(\frac{d}{2}\right)}\left|\frac{1}{8d} \eta_{+}\beta_{+}\eta_-\beta_- \left(\frac{3\beta_+}{4}\right)^{\frac{d}{2}-1}\right| E\left(\frac{3\beta_-}{2}, n, \frac{d}{2}-1\right)
\end{align*}
as $|z|\mapsto\left|E(z,n,\nu)\right|$ is an increasing function, thanks to Lemma~\ref{lem: bessel}.

Similar arguments as for Lemma~\ref{lem: appendix} show that $\mathring{\mathcal{N}}^{n}(\mu)$ is decreasing on $[2.5,3]$. Table~\ref{table:arb d_domain_numerator} summarises the values $\mathring{\mathcal{N}}^{n}(2.5)$ and $\text{Err}^{n}_{\mathcal{N}}$ for $n = 6$. By Lemma \ref{Lemma: bound_estimator}, we can conclude that
\begin{eqnarray*}
\sup_{\lambda \in [2.5, 3]}\mathcal{N}(\lambda) \leq \mathring{\mathcal{N}}^{n}(2.5) + \text{Err}^{n}_{\mathcal{N}} < 0.
\end{eqnarray*}

\begin{table}[ht]
\caption{The numerator $\mathcal{N}(\mu)$ is non-vanishing on $]2.5,3[$ for general domains} 
\centering 
\begin{tabular}{c c c} 
\hline\hline 
Dimension & Value for $\mathring{\mathcal{N}}^n(2.5)$ & Value for $\text{Err}^n_{\mathcal{N}}$ \\ [0.5ex] 
\hline 
3 & -3.487949  &  0.366411\\ 
4 & -5.367159 & 0.317482\\
5 &-6.082985  & 1.552336 \\
6 & -5.465016 &  0.811930 \\
7 & -3.964234 &  2.802239\\
8 & -2.193441 & 1.131338 \\ [1ex] 
\hline 
\end{tabular}
\label{table:arb d_domain_numerator}
\end{table}

Thus we have proved the second claim (2).   \end{proof}

\bigskip

\noindent
{\bf Acknowledgement}. HN is supported by the EPSRC under the program grant EP/S026347/1 and the Alan Turing Institute under the EPSRC grant EP/N510129/1. 
The authors extend their gratitude to Terry Lyons, Weijun Xu, and Guangyu Xi for kind communications. SL also thanks the mathematics department at Rice University, Houston for the nice working atmosphere, where this work was initiated when SL worked as a G.~C. Evans Instructor. Both authors are deeply indebted to Horatio Boedihardjo for very insightful and constructive discussions.

\bibliographystyle{unsrt}
\bibliography{LiNiApril2021}
\end{document}